\documentclass{article}
\usepackage{amsmath,amssymb,amscd,enumerate,bbm,geometry}

\usepackage{tikz}
\usepackage[latin1]{inputenc}
\usetikzlibrary{arrows,calc,positioning,trees,shadows}

\newcommand{\email}[1]{{{\em Email}: {\tt #1}}}
\newcommand{\tmop}[1]{\ensuremath{\operatorname{#1}}}

\newtheorem{definition}{Definition}
\newtheorem{corollary}{Corollary}

\newtheorem{proposition}{Proposition}
\newenvironment{proof}{
  \noindent\textbf{Proof}\ }{\hspace*{\fill}
  \begin{math}\Box\end{math}\medskip}
\newtheorem{notation}{Notation}

\newcommand{\tmmathbf}[1]{\ensuremath{\boldsymbol{#1}}}
\newtheorem{theorem}{Theorem}
\newtheorem{lemma}{Lemma}
\newtheorem{remark}{Remark}
\newtheorem{example}{Example}
\def\CC{{\mathbbm C}}

\newcommand{\dsp}{\displaystyle}

\def\shuff#1#2{\mathbin{
      \hbox{\vbox{\hbox{\vrule \hskip#2 \vrule height#1 width 0pt}\hrule}\vbox{\hbox{\vrule \hskip#2 \vrule height#1 width 0pt\vrule }\hrule}}}}

\def\qshuffl{{\mathchoice{\shuff{5pt}{3.5pt}\hspace{-2.9mm}-}{\shuff{5pt}{3.5pt}\hspace{-2.9mm}-}
{\shuff{3pt}{2.6pt}\hspace{-2.2mm}-}{\shuff{3pt}{2.6pt}\hspace{-2.2mm}-}}}
\def\qshuffle{\,\qshuffl\,}
\newenvironment{arb}{\begin{tikzpicture}[baseline,scale=0.5,level distance=7mm,level 1/.style={sibling distance=10mm},level 2/.style={sibling distance=5mm},level 3/.style={sibling distance=3mm},grow=down, font=\scriptsize]
\tikzstyle{ve}=[draw,circle,inner sep=1pt,fill] 
\tikzstyle{vv}=[draw,circle,inner sep=1pt] 
\tikzstyle{vee}=[minimum size=0pt ,inner sep=0pt]}{\end{tikzpicture}}
\newenvironment{arbb}{\begin{tikzpicture}[baseline,scale=0.5,level distance=7mm,level 1/.style={sibling distance=25mm},level 2/.style={sibling distance=5mm},level 3/.style={sibling distance=3mm},grow=down, font=\scriptsize]
\tikzstyle{ve}=[draw,circle,inner sep=1pt,fill] 
\tikzstyle{vv}=[draw,circle,inner sep=1pt] 
\tikzstyle{vee}=[minimum size=0pt ,inner sep=0pt]}{\end{tikzpicture}}

\begin{document}

\title{Renormalization: a quasi-shuffle approach.} 
\author{Fr\'ed\'eric Menous\footnote{
Laboratoire de Math\'ematiques d'Orsay, 
Univ. Paris-Sud, CNRS, Universit\'e Paris-Saclay, 
91405 Orsay, France. \email{Frederic.Menous@math.u-psud.fr}}\  and Fr\'ed\'eric Patras\footnote{
Laboratoire J.A. Dieudonn\'e,
         		Universit\'e de la C\^ote d'Azur,
         		CNRS, UMR 7531, 
         		Parc Valrose,
         		06108 Nice Cedex 2, France.
\email{Frederic.PATRAS@unice.fr}}}
% \URLADDR{WWW-MATH.UNICE.FR/~PATRAS}

\date{}

\maketitle

\abstract{In recent years, the usual BPHZ algorithm for renormalization in perturbative quantum
  field theory has been interpreted, after dimensional regularization, as a
  Birkhoff decomposition of characters on the Hopf algebra
  of Feynman graphs, with values in a Rota-Baxter algebra of amplitudes.
  We associate in this paper to any such algebra a universal semigroup (different in nature from the Connes-Marcolli "cosmical Galois group"). Its action on the physical amplitudes associated to Feynman graphs produces the expected operations: Bogoliubov's preparation map, extraction of divergences, renormalization.
 In this process a key role is played by commutative and noncommutative quasi-shuffle bialgebras whose universal properties are instrumental in encoding the renormalization process.}

\section*{Introduction.}
In the early 2000s, the usual BPHZ algorithm for renormalization in perturbative quantum
  field theory has been interpreted, after dimensional regularization, as a
  Birkhoff decomposition of characters on the Hopf algebra
  of Feynman graphs, with values in a Rota-Baxter algebra of amplitudes \cite{ck1,ck2,guo}. 
  This idea was later shown to be meaningful in a broad variety of contexts: in the theory of dynamical systems, in analysis and numerical analysis (Rayleigh-Schr\"odinger series) or, more recently, in the theory of regularity structures and the study of very irregular stochastic differential equations or stochastic partial differential equations, see e.g. \cite{men,men06,men13,hai,hz}
  
In this context, P. Cartier suggested the existence of a hidden universal symmetry group (the "cosmical Galois group") that would underlie renormalization. Using geometrical tools such as universal singular frames, Connes and Marcolli constructed a candidate group in 2004 \cite{CM}. Their construction was translated in the langage of Hopf algebras in \cite{KGP} and the group shown to coincide with the prounipotent group of group-like elements in the completion with respect to the grading of the descent algebra -a Hopf algebra that, as an algebra, is the free associative algebra generated by the Dynkin operators \cite{Pat}.

However, the action of this group or of the descent algebra on the Hopf algebras of Feynman diagrams showing up in pQFT does not actually perform renormalization. It captures nicely certain phenomena related to Lie theory and the behaviour of the Dynkin operators: for example, the structure of certain renormalization group equations and the algebraic properties of beta functions (see the original article by Connes and Marcolli \cite{CM} and the detailed algebraic and combinatorial analysis of these phenomena in \cite{patrasC}. Further insights on the role of (generalized) Dynkin operators in the theory of differential equations can be found in \cite{MP}). However, the group and the descent algebra act on Feynman diagrams and do not encode operations that occur at the level of the target algebra of amplitudes. They fail therefore to
capture typical renormalization operations such as projections on divergent or regular components of amplitudes. Substraction maps, for example, cannot be encoded in it, and neither are more advanced operations such as the construction of the counterterm.

In the present article, we follow a different approach that complements Connes-Marcolli's and its Hopf algebraic and combinatorial interpretation by showing show how a semigroup of operators can be associated to the algebra of coefficients of a given regularization and renormalization scheme in pQFT. Its construction relies heavily on the universal properties of commutative and noncommutative quasi-shuffle algebras. This semigroup acts in a natural way on regularized amplitudes and perform the expected operations: preparation map, extraction of counterterms, renormalization. Notice that many of our results and constructions do not require the algebra of coefficients to be commutative.

Let us sketch up the ideas and results.
Concretely we deal with conilpotent bialgebras $H = k \oplus H^+$. These bialgebras are Hopf algebras and the
coalgebra structure on $H$ induces a convolution product on the space
$\mathcal{L}( H, A )$ of linear morphisms from $H$ to an associative algebra
$A$. If $A$ is unital, then the subset $\mathcal{U}( H, A )$ of linear
morphisms that send the unit $1_H$ of $H$ on the unit $1_A$ of $A$ is a group for the convolution and, if A is
commutative, the subset $\mathcal{C}( H, A )$ of characters (i.e. algebra morphisms)
is a subgroup of $\mathcal{U}( H, A )$. 

In pQFT, the algebra $A$ is often called the algebra of (regularized) amplitudes, and we will often use this terminology.
In this context,
the renormalization process equips the target  unital algebra $A$ with
a projection operator $p_+$ such that
$$A = \tmop{Im} p_+ \oplus \tmop{Im} p_- = A_+ \oplus A_-,$$ where $p_- =
  \tmop{Id} - p_+$ and  $A_+$ and $A_-$ are subalgebras.
Here, $p_-$ should be thought of as a projection on the "divergent part", so that $p_+$ substract divergences. For example, in dimensional regularization, $A$ identifies with the algebra of Laurent series, $\CC[[\varepsilon,\varepsilon^{-1}]$, and $p_-$ (resp. $p_+$) is the projection on $\varepsilon^{-1} \CC[\varepsilon^{-1}]$ (resp. $\CC[[\varepsilon ]]$).
As was first observed by  Ebrahimi-Fard, building on previous results by Brouder and Kreimer, these data define a Rota-Baxter algebra structure on $A$ and $\mathcal{L}( H, A )$.

The choice of the subtraction operator is not always unique --for example when using momentum subtraction schemes. How this phenomenon impacts the combinatorics and Rota-Baxter structures was investigated in \cite{EFP2}. Although we do not investigate it further here, the tools we develop in the present article should be useful in that context since they put forward the idea that one should study for its own the combinatorial structure of the target algebra of amplitudes $A$, independently of the choice of a particular subtraction map $p_+$.

It is then well-know that, given $p_+$, there exists a unique
Birkhoff decomposition of any morphism $\varphi \in
\mathcal{U}( H, A )$
\[ \varphi_- \ast \varphi = \varphi_+ \hspace{2em} \varphi_+, \varphi_- \in
   \mathcal{U}( H, A ) \]
where $\varphi_+ ( H^+ ) \subset A_+$ and $\varphi_- ( H^+ ) \subset A_-$.
Moreover, if $A$ is commutative, this decomposition is defined in the subgroup
$\mathcal{C}( H, A )$. The classical proofs of this result are recursive, using the
filtration on $H$ (they rely ultimately on the Bogoliubov recursion \cite{emp}). 

We propose to develop here a ``universal'' framework to handle the combinatorics of renormalization and to give in this framework explicit, and in some sense universal,
formulas for $\varphi_+$ and $\varphi_-$.
To do so, we consider the quasi-shuffle Hopf algebra
$QSh(A)$ over an algebra $A$, that is, the standard tensor coalgebra over
$A$ equipped with the quasi-shuffle (or stuffle) product. Using the properties of the functor $QSh$ (including the surprising property, for any Hopf algebra $H$ to be canonically embedded into $QSh(H^+)$), we compute then the inverse and the Birkhoff decomposition of a
fundamental element $j \in \mathcal{U}( QSh(A), A )$ defined by
\[ j ( 1 ) = 1_A, \hspace{1em} j ( a_1 ) = a_1, \hspace{1em} j ( a_1
   \otimes \ldots \otimes a_s ) = 0 \hspace{1em} \tmop{if} \hspace{1em} s \geq
   2 .\]
We show then the existence of an action of $\mathcal{U}(
QSh(A), A )$ on $\mathcal{U}( H, A )$. More precisely we define a map
$$\mathcal{U}( QSh(A), A ) \times \mathcal{U}( H, A ) \rightarrow
\mathcal{U}( H, A )$$
$$(f,\varphi)\mapsto f\odot\varphi,$$ such that
\[  j\odot \varphi = \varphi \hspace{1em} \tmop{and} \hspace{1em} ( f \ast
   g)\odot \varphi =  ( f\odot \varphi ) \ast ( g\odot \varphi ), \]
and obtain explicit formulas such as:
\begin{enumerate}
  \item If $j^{\ast - 1}$ is the inverse of $j$, then $\varphi^{\ast - 1} = 
   j^{\ast - 1}\odot \varphi $.
  
  \item If $j_- \ast j = j_+$ (Birkhoff decomposition), then $\varphi_- \ast \varphi
  = \varphi_+$ where $\varphi_{\pm} =  j_{\pm}\odot \varphi$.
\end{enumerate}

The article is organized as follows. After a preliminary section fixing notations and recalling general properties of Hopf algebras,
section \ref{sect:1} analyses the algebraic properties of algebras of regularized amplitudes and explains how they give rise to quasi-shuffle algebra structures. Section \ref{sect:2} introduces Hoffman's quasi-shuffle functor (i.e. the notion of quasi-shuffle algebra over an algebra -in the commutative case, it is the left adjoint to the forgetful functor from quasi-shuffle algebras to commutative algebras). Section \ref{sect:3} investigates its categorical properties, including a surprising right adjoint property (Thm. \ref{mainthm}). Section \ref{sect:4} studies, using these techniques, the map $j$ (mapping a cofree coalgebra to its cogenerating vector space). This is the key to latter applications to renormalization which are the purpose of Section \ref{sect:5}, as well as the construction, for each algebra of amplitudes, of a ``universal semigroup'' in which the operations characteristic of renormalization are encoded. The last two sections survey various applications, in particular to Dynamics and Analysis.

\ \par \ \par
We acknowledge support from the CARMA grant ANR-12-BS01-0017,
"Combinatoire
Alg\'ebrique, R\'esurgence, Moules et Applications"and the CNRS GDR "Renormalisation".
We thank warmly K. Ebrahimi-Fard, from whom we learned some years ago already the meaningfulness of Rota--Baxter algebras and their links with quasi--shuffle algebras.

\section{Notation and Hopf algebra fundamentals}\label{section:0}
Everywhere in the article, algebraic structures are defined over a fixed ground field $k$ of characteristic $0$.
We fix here the notations relative to bialgebras and Hopf algebras, following {\cite{fig}} (see also
\cite{Cartier2}, {\cite{maj}} and {\cite{sweed}}) and refer to these articles and surveys for details and  generalities on the subject.
Recall that a bialgebra $B$ is an associative algebra with unit and a coassociative coalgebra with counit such that the product is a morphism of coalgebras (or, equivalently, the coproduct is a morphism of algebras). We will usually write $m$ the product, $\Delta$ the coproduct, $u:k\to B$ the unit and $\eta:B\to k$ the counit. When ambiguities might arise we put an index (and denote e.g. $m_B$ the product instead of $m$).

 We use freely the Sweedler notation and write
\begin{equation}
  \Delta h = \sum h_{( 1 )} \otimes h_{( 2 )}.
\end{equation}
Thanks to coassociativity, we can define recursively and without any
ambiguity the linear morphisms $\Delta^{[ n ]} : B \rightarrow B^{\otimes n}$
($n \geq 1$) by $\Delta^{[ 1 ]} = \tmop{Id}$ and, for $n \geq 1$,
\begin{equation}
  \Delta^{[ n + 1 ]} = ( \tmop{Id} \otimes \Delta^{[ n ]} ) \circ \Delta = (
  \Delta^{[ n ]} \otimes \tmop{Id} ) \circ \Delta = ( \Delta^{[ k ]} \otimes
  \Delta^{[ n + 1 - k ]} ) \circ \Delta \hspace{1em} ( 1 \leq k \leq n )
\end{equation}
and write
\begin{equation}
  \Delta^{[ n ]} h = \sum h_{( 1 )} \otimes \ldots \otimes h_{( n )}
\end{equation}
In the same way, for $n \geq 1$, we define $m^{[ n ]} : B^{\otimes^n}
\rightarrow B$ by $m^{[ 1 ]} = \tmop{Id}$ and
\begin{equation}
  m^{[ n + 1 ]} = m \circ ( \tmop{Id} \otimes m^{[ n ]} ) = m \circ ( m^{[ n
  ]} \otimes \tmop{Id} )
\end{equation}

The reduced coproduct $\Delta'$ on $H^+:=Ker \ \eta$ is
defined by
\begin{equation}
  \Delta' h = \Delta h - 1 \otimes h - h \otimes 1
\end{equation}
Its iterates (defined as for $\Delta$) are written $\Delta'^{[ n ]}$. A bialgebra is conilpotent (or, more precisely, locally conilpotent) is for any $h\in H^+$ there exists a $n\geq 1$ (depending on $h$) such that $\Delta'^{[ n ]}(h)=0.$

A bialgebra $H$ is a Hopf algebra if there exists an antipode $S$, that is to
say a linear map $S : H \rightarrow H$ such that :
\begin{equation}
  m \circ ( \tmop{Id} \otimes S ) \circ \Delta = m \circ ( S \otimes \tmop{Id}
  ) \circ \Delta = u \circ \eta : H \rightarrow H
\end{equation}
In this article, we will consider only conilpotent bialgebras, which are
automatically Hopf algebras.

Given a connected bialgebra $H$ and an algebra $ A$ with product $ m_A$ and unit $ u_A$, the coalgebra
structure of $H$ induces an associative convolution product on the vector
space $\mathcal{L}( H, A )$ of $k$--linear maps :
\begin{equation}
  \forall ( f, g ) \in \mathcal{L}( H, A ) \times \mathcal{L}( H, A ),
  \hspace{1em} f \ast g = m_A \circ ( f \otimes g ) \circ \Delta
\end{equation}
with a unit given by $u_A \circ \eta$, such that $(\mathcal{L}( H, A ), \ast,
u_A \circ \eta )$ is an associative unital algebra.

\begin{lemma}
  Let $H$ be a conilpotent bialgebra (and therefore a Hopf algebra) and set
  \begin{equation}
    \mathcal{U}( H, A ) = \{ f \in \mathcal{L}( H, A ) \hspace{1em} ;
    \hspace{1em} f ( 1_H ) = 1_A \}
  \end{equation}
  then $\mathcal{U}( H, A )$ is a group for the convolution product.
\end{lemma}

\begin{proof}
  $\mathcal{U}( H, A )$ is obviously stable for the convolution product.
  Following {\cite{fig}}, we will remind why any element $f \in \mathcal{U}(
  H, A )$ as a unique inverse $f^{\ast - 1}$ in $\mathcal{U}( H, A )$. 
  One can write formally
  \begin{equation}
    f^{\ast - 1} = ( u_A \circ \eta - ( u_A \circ \eta - f ) )^{\ast - 1} =
    u_A \circ \eta + \sum_{k \geq 1} ( u_A \circ \eta - f )^{\ast k}
  \end{equation}
  This series seems to be infinite but, because of the conilpotency assumption, for any $h \in H'$
  \begin{equation}
    ( u_A \circ \eta - f )^{\ast k} ( h )=( - 1 )^k m_A^{[ k ]} \circ f^{\otimes k} \circ
    \Delta'^{[ k ]} ( h ) \label{eq15}
  \end{equation}
  vanishes for $k$ large enough.
  \end{proof}
  
When this result is applied to $\tmop{Id} : H \rightarrow H \in
\text{$\mathcal{U}( H, H )$}$, then its convolution inverse is the antipode
$S$ (this is the usual way of proving that any conilpotent bialgebra is a Hopf algebra). 

\begin{notation}
  If $B \subset A$ is a subalgebra of $A$ which is not unital, then we write
  \[ \mathcal{U}( H, B ) = \{ f \in \mathcal{L}( H, A ) \hspace{1em} ;
     \hspace{1em} f ( 1_H ) = 1_A \hspace{1em} \tmop{and} \hspace{1em} f ( H^+
     ) \subset B \} \]
  This is a subgroup of $\mathcal{U}( H, A )$.
\end{notation}

Let now $\mathcal{C}( H, A )$ be the subset of $\mathcal{L}( H, A )$ whose elements
are algebra morphisms (also called characters over $A$). Of course,
\[ \mathcal{C}( H, A ) \subset \text{$\mathcal{U}( H, A )$} \]
but this shall not be a subgroup: if $A$ is not commutative, there is no
reason why it should be stable for the convolution product. Nonetheless if $A$ is commutative, the product from $A\otimes A$ to $A$ is an algebra map: it follows that the convolution of algebra morphisms is an algebra morphism and $\mathcal{C}( H, A )$ is a subgroup of
$\text{$\mathcal{U}( H, A )$}$.

Moreover if $f
\in \text{$\mathcal{U}( H, A )$}$ is an algebra map, then its inverse $f^{\ast
- 1}$ in $\mathcal{U}( H, A )$ is an antialgebra map given by $f^{\ast - 1} = f \circ S$ :
\begin{equation}
  \begin{array}{ccc}
    f \ast f \circ S & = & m_A \circ ( f \otimes f \circ S ) \circ \Delta\\
    & = & m_A \circ ( f \otimes f ) \circ ( \tmop{Id} \otimes S ) \circ
    \Delta\\
    & = & f \circ m \circ ( \tmop{Id} \otimes S ) \circ \Delta\\
    & = & f \circ u \circ \eta\\
    & = & u_A \circ \eta ,
  \end{array}
\end{equation}
where we recall that the antipode is an antialgebra morphism:
$$S(gh)=S(h)S(g).$$

\section{From renormalization to quasi-shuffle algebras}\label{sect:1}
The fundamental ideas of renormalization in pQFT were already alluded at in the introduction, we recall them very briefly and refer to textbooks for details (this first paragraph is mainly motivational, we will move immediately after to an algebraic framework that can be understood without mastering the quantum field theoretical background). Starting from a given quantum field theory, one expands perturbatively the quantities of interest (such as Green's functions). This expansion is indexed by Feynman diagrams, and to each of these diagrams is associated a quantity computed by means of certain integrals. Very often, these integrals are divergent and need to be regularized and renormalized. 
Typically, a quantity such as 
$$
    \phi(c):=\int_{0}^{\infty} \frac{dy}{y+c}
$$ is divergent, but becomes convergent up to the introduction of an arbitrary small regularizing parameter $\varepsilon$ (for dimensional reasons, one also introduces a mass term $\mu$)
$$
    \phi(c;\varepsilon):=\int_{0}^{\infty} \frac{\mu^{\varepsilon}d y}{(y+c)^{1+\varepsilon}}
                                =\frac{1}{\varepsilon} + \log(\mu/c) +O(\varepsilon).$$
In that toy model case, close to the dimensional regularization method, the "regularized amplitude" $\phi(c;\varepsilon)$ lives in $A=\CC[[\varepsilon,\varepsilon^{-1} ]$ and is renormalized by removing the divergency $\frac{1}{\varepsilon}$ (the component of the expansion in $\varepsilon^{-1}\CC[\varepsilon^{-1}]$). 

These ideas are axiomatized using the notion of Rota--Baxter algebras as follows.
Following {\cite{kur}}, let $p_+$ an idempotent of $\mathcal{L}( A, A )$ where
$A$ is a unital algebra (in our toy model example, $p_+$ would stand for the projection on $\CC[[\varepsilon ]]$). If we have for $x, y$ in $A$ :
\begin{equation}
  p_+ ( x ) p_+ ( y ) + p_+ ( x y ) = p_+ ( x p_+ ( y ) ) + p_+ ( p_+ ( x ) y
  ) ),
\end{equation}
then $p_+$ is a Rota-Baxter operator, $( A, p_+ )$ is a Rota-Baxter algebra
and if $p_- = \tmop{Id} - p_+$, $A_+ = \tmop{Im} p_+$ and $A_- = \tmop{Im}
p_-$ then
\begin{itemize}
  \item $A = A_+ \oplus A_-$.
  
  \item $p_-$ satisfies the same relation.
  
  \item $A_+$ and $A_-$ are subalgebras.
\end{itemize}
Conversely if $A = A_+ \oplus A_-$ and $A_+$ and $A_-$ are subalgebras, then
the projection $p_+$ on $A_+$ parallel to $A_-$ defines a Rota-Baxter algebra
$( A, p_+ )$.

The idempotency condition is not required to define a Rota--Baxter algebra. In general:

\begin{definition}
A Rota--Baxter (RB) algebra is an associative algebra $A$ equipped with a linear endomorphism $R$ such that
$$\forall x,y\in A, R(x)R(y)=R(R(x)y+xR(y)- xy).$$
It is an idempotent RB algebra if $R$ is idempotent (in that case we will set $p_+:=R$ to emphasize that we are in the framework typical for renormalization). It is a commutative Rota--Baxter algebra if it is commutative as an algebra.
\end{definition}

The notion of Rota--Baxter algebra is actually slightly more general: a Rota--Baxter algebra of weight $\theta$ is defined by the identity
$$\forall x,y\in A, R(x)R(y)=R(R(x)y+xR(y)+\theta xy).$$
We restrict here the definition to the weight $-1$ case, which is the one meaningful for renormalization.

Using Rota--Baxter algebras of amplitudes, the principle of renormalization in physics can be formulated algebraically in the following
way.

\begin{proposition}
  Let $H$ be a conilpotent bialgebra and $( A, p_+ )$ an idempotent Rota-Baxter algebra (so that $A=A_-\oplus A_+$). Then
  for any $\varphi \in \mathcal{U}( H, A )$ there exists a unique pair $(
  \varphi_+, \varphi_- ) \in \mathcal{U}( H, A_+ ) \times \mathcal{U}( H, A_-
  )$ such that
  \begin{equation}
    \varphi_- \ast \varphi = \varphi_+
  \end{equation}
  Moreover, if $A$ is commutative and $\varphi$ is a character, then
  $\varphi_+$ and $\varphi_-$ are also characters. This factorization is
  called the Birkhoff decomposition of $\varphi$.
\end{proposition}

\begin{proof}Let us postpone the assertion on characters and prove the existence and unicity -notions such as the one of Bogoliubov's preparation map will be useful later.
 As $A_+$ and $A_-$ are
  subalgebras of $A$, $\mathcal{U}( H, A_+ )$ and $\mathcal{U}( H, A_- )$ are
  subgroups of $\mathcal{U}( H, A )$.
  If such a factorization exists, then it is unique : If $\varphi =
  \varphi_-^{\ast - 1} \ast \varphi_+ = \psi_-^{\ast - 1} \ast \psi_+$, then
  \[ \phi = \psi_+^{} \ast \varphi^{\ast - 1}_+ = \psi_- \ast \varphi_-^{\ast
     - 1} \in \mathcal{U}( H, A_+ ) \cap \mathcal{U}( H, A_- ) \]
  thus for $h \in H^+$, $\phi ( h ) \in A_+ \cap A_- = 0$. We finally get that
  \[ \psi_+^{} \ast \varphi^{\ast - 1}_+ = \psi_- \ast \varphi_-^{\ast - 1} =
     u_A \circ \eta \]
  and $\varphi_+ = \psi_+$, $\varphi_- = \psi_-$.
  
  Let us prove now that the factorization exists. Let $\varphi \in
  \mathcal{U}( H, A )$, we must have $\varphi_+ ( 1_H ) = \varphi_- ( 1_H ) =
  1_A$. Let $\bar{\varphi} \in \mathcal{U}( H, A )$ the Bogoliubov preparation
  map defined recursively on the increasing sequence of vector spaces $H^+_n:=Ker {\Delta'}^{[n]}$ ($n \geq 1$) by
  \begin{equation}\label{eq:Bog}
    \bar{\varphi} ( h ) = \varphi ( h ) - m_A \circ ( p_- \otimes \tmop{Id}
    ) \circ ( \bar{\varphi} \otimes \varphi ) \circ \Delta' ( h )
  \end{equation}
  (since $H$ is conilpotent, $H^+=\cup_nH^+_n$).
  Now if $\varphi_+$ and $\varphi_-$ are the elements of $\mathcal{U}( H, A )$
  defined on $H^+$ by
  \[ \varphi_+ ( h ) = p_+ \circ \bar{\varphi} ( h ) \hspace{1em},
     \hspace{1em} \varphi_- ( h ) = - p_- \circ \bar{\varphi} ( h )
     \hspace{1em} ( \bar{\varphi} ( h ) = \varphi_+ ( h ) - \varphi_- ( h ) ),
  \]
  then
  \[ \text{} \varphi_+ \in \mathcal{U}( H, A_+ ) \hspace{1em}, \hspace{1em}
     \varphi_- \in \text{$\mathcal{U}( H, A_- )$} \hspace{1em}, \hspace{1em}
     \varphi_- \ast \varphi_{} = \varphi_+ \]
 
\end{proof}

We turn now to another algebraic structure, induced by the one of RB algebras, but weaker --the one we will be concerned later on: quasi-shuffle algebras. Concretely, the target algebras of amplitudes (such as the algebra of Laurent series) happen to be quasi-shuffle algebras, whereas the algebras of linear forms on Feynman diagrams with values in a commutative RB algebra of amplitudes happen to be noncommutative quasi-shuffle algebras.

Indeed, a RB algebra is always equipped with an associative product, the RB double product $\star$, defined by:
\begin{equation}
x\star y:=R(x)y+xR(y)-xy
\end{equation}
so that: $R(x)R(y)=R(x\star y)$.
Setting $x\prec y:=xR(y),\ x\succ y:=R(x)y$, one gets 
$$(x  y)\prec z=xyR(z)=x (y\prec z),$$
$$ (x\prec y)\prec z=xR(y)R(z)=x\prec (y\star z),$$
$$(x\succ y)\prec z=R(x)yR(z)=x\succ (y\prec z),$$
and so on. 
These observations give rise to the axioms of noncommutative quasi-shuffle  algebras (NQSh, also called tridendriform, algebras). On an historical note, we learned recently from K. Ebrahimi-Fard that the following axioms and relations seem to have first appeared in the context of stochastic calculus, namely in the work of Karandikar in the early 80's on matrix semimartingales, see e.g. \cite{Kandihar}. See also \cite{fpQSh} for details and other references. 

\begin{definition}
A noncommutative quasi-shuffle algebra (NQSh algebra) is a nonunital associative algebra (with product written $\bullet$) equipped with two other products $\prec, \succ$
such that, for all $x,y,z\in A$:
\begin{align}
\label{E1}(x\prec y)\prec z=x\prec(y\star z),&\ \ (x\succ y)\prec z=x\succ(y\prec z)\\
\label{E3}(x\star y) \succ z=x\succ (y\succ z),& \ \ (x\prec y)\bullet z=x\bullet (y\succ z)\\
\label{E6}(x \succ y)\bullet z=x\succ (y\bullet z),& \ \ (x\bullet y)\prec z=x\bullet (y\prec z).
\end{align}
where $x\star y:=x\prec y+x\succ y+x\bullet y$.
 \end{definition}

Notice that
$
(x\bullet y)\bullet z= x\bullet (y\bullet z)$ and $(\ref{E1})+(\ref{E3})+(\ref{E6})$ imply the associativity of $\star$:
\begin{equation}
(x\star y)\star z=x\star (y\star z).
\end{equation}

When the RB algebra is commutative, the relations between the three products $\prec , \succ , \bullet$ simplify (since $x\prec y=xR(y)=y\succ x$) and one arrives at the definition:

\begin{definition}
A quasi-shuffle (QSh) algebra $A$ is a nonunital commutative algebra (with product written $\bullet$) equipped with another product $\prec$ such that
\begin{align}
\label{QS1}(x\prec y)\prec z&=x\prec(y\star z)\\
\label{QS2}(x\bullet y)\prec z&=x\bullet(y\prec z).
\end{align}
where $x\star y:=x\prec y+y\prec x+x\bullet y$. 
 \end{definition}
 
 We also set for further use $x\succ y:=y\prec x$ (this makes a QSh algebra a NQSh algebra).
The product $\star$ is automatically associative and commutative and defines another commutative algebra structure on $A$. 

It is sometimes convenient to equip NQSh and QSh algebras with a unit. The phenomenon is exactly similar to the case of shuffle algebras \cite{schutz}. Given a NQSh algebra, one sets $B:=k\oplus A$, and the products $\prec$, $\succ$, $\bullet$ have a partial extension to $B$ defined by, for $x\in A$:
$$1\bullet x=x\bullet 1:=0,  \ 1\prec x:=0,\ x\prec 1:=x, \ 1\succ x:=x, \ x\succ 1:=0.$$
The products $1\prec 1$, $1\succ 1$ and $1\bullet 1$ cannot be defined consistenly, but one sets $1\star 1:=1$, making $B$ a unital commutative algebra for $\star$.
The categories of NQSh/QSh and unital NQSh/QSh algebras are equivalent (under the operation of adding or removing a copy of the ground field).

Formally, the relations between RB algebras and NQSh algebras are encoded by the Lemma:
\begin{lemma}
The identities $x\prec y:=xR(y),\ x\succ y:=R(x)y, x\bullet y:=xy$ induce a forgetful functor from RB algebras to NQSh algebras, resp. from commutative RB algebras to QSh algebras.
\end{lemma}

We already alluded to the fact that, in a given quantum field theory, the set of linear forms from the linear span of Feynman diagrams (or equivalently algebra maps from the polynomial algebra they generate) to a commutative RB algebra of amplitudes carries naturally the structure of a noncommutative RB algebra. In the context of QSh algebras, this result generalizes as follows:

\begin{proposition}\label{convol}
Let $C$ be a (coassociative) coalgebra with coproduct $\Delta$ and $A$ be a NQSh algebra. Then the set of linear maps $Hom(C,A)$ is naturally equipped with the structure of a NQSh algebra by the products:
$$f\prec g (c):=f(c^{(1)})\prec g(c^{(2)}),$$
$$f\succ g (c):=f(c^{(1)})\succ g(c^{(2)}),$$
$$f\bullet g (c):=f(c^{(1)})\bullet g(c^{(2)}),$$
where we used Sweedler's notation $\Delta(c)=c^{(1)}\otimes c^{(2)}$.
\end{proposition}

The proposition follows from the fact that the relations defining NQSh algebras are non-symmetric (in the sense that they do not involve permutations:
for example, in the equation
$(x\prec y)\prec z=x\prec(y\star z)$, the letters $x,y,z$ appear in the same order in the left and right hand side, and similarly for the other defining relations). 

\section{The quasi-shuffle Hopf algebra $QSh(A)$.}
\label{sect:2}

For details on the constructions in this section, we refer the reader to
{\cite{hof,fpQSh,hi}}.
Let $A$ be an associative algebra. We write $QSh(A)$ for the graded vector space
$QSh(A) = \bigoplus_{n \geq 0} QSh(A)_{ n }=k\oplus\bigoplus_{n \geq 1} QSh(A)_{ n }=:k\oplus QSh^+(A)$ where, for $n
\geq 1$, $QSh(A)_{n} = A^{\otimes n}$ and  $QSh(A)_{0} =
k$ (notice that when $A$ is unital, one has to distinguish between $1\in k=QSh(A)_{0}$ and $1_A\in A\subset QSh(A)_1$). We denote $l (\tmmathbf{a}) = n$ the length of an element
$\tmmathbf{a}$ of $QSh(A)_{n}$. 

For convenience, an element
$\tmmathbf{a}= a_1 \otimes \ldots \otimes a_n$ of $QSh(A)$ will be called
a word and will be written $a_1\dots a_n$ (it should not be confused with the product of the $a_i$ in $A$). We will reserve the tensor product notation for the tensor product of elements of $QSh(A)$ (so that for example, $a_1a_2\otimes a_3\in QSh(A)_2\otimes QSh(A)_1$). Also, we distinghish between the concatenation product of words (written $\cdot$: $a_1a_2a_3\cdot b_1b_2=a_1a_2a_3b_1b_2$) and the product in $A$ by writing $a\cdot_A b$ the product of $a$ and $b$ in $A$ (whereas $a\cdot b$ would stand for the word $ab$ of length 2).

The graded vector space $QSh^+(A)$ (resp. $QSh(A)$) is given  a graded (resp. unital) NQSh algebra structure by induction on the length of tensors such that for all $a,b \in A$, for all $v,w \in QSh(A)$:
$$
av\prec bw=a(v\star bw),$$
$$av \succ bw=b(av \star w),$$
$$av \bullet bw=(a._Ab)(v \star w),$$
where $\qshuffle:=\star=\prec+\succ+\bullet$ is usually called the quasi-shuffle (or stuffle) product (by definition: $\forall v\in QSh(A), 1\qshuffle v=v=v\qshuffle 1$). 
Notice that this product $\qshuffle$ can be defined directly by the two equivalent inductions
$$av\qshuffle bw:=a(v\qshuffle bw)+b(av\qshuffle w)+a\cdot_Ab (v\qshuffle w)$$
or
$$va\qshuffle wb:=(v\qshuffle bw)a+(av\qshuffle w)b+ (v\qshuffle w)a\cdot_Ab.$$

When $A$ is commutative, $QSh(A)$ is a unital quasi-shuffle algebra

For example :
\begin{equation}
  a_1 a_2 \qshuffle b  = a_1 
  a_2  b + a_1  b  a_2 + b a_1  a_2
  + a_1 ( a_2\cdot_A b)+ (a_1\cdot_A b)a_2
\end{equation}
Notice at last that, under the action of the four products $\prec, \succ,\star,\bullet$, the image of
$ QSh(A)_{r } \otimes QSh(A)_{
  s } $ is contained in $ \bigoplus_{t = \max ( r, s )}^{r + s} QSh(A)_{ t }$

One can also define :
\begin{itemize}
  \item a counit $\eta : QSh(A) \rightarrow k$ by
  $\eta ( 1 ) := 1$ and for $s \geq 1$, $\eta
  ( a_1 \ldots  a_s ) = 0$,
  
  \item a coproduct (called deconcatenation coproduct) $\Delta : QSh(A) \rightarrow
  QSh(A) \otimes QSh(A)$ such that
  $\Delta( 1) = 1\otimes 1$
  and for $s \geq 1$ and $\tmmathbf{a}= a_1  \ldots  a_s \in
  QSh(A)_{s}$,
  \begin{equation}
    \Delta ( \tmmathbf{a} ) = \tmmathbf{a} \otimes
    1 + 1 \otimes\tmmathbf{a} + \sum_{r = 1}^{s -
    1} ( a_1  \ldots  a_r ) \otimes( a_{r + 1}
     \ldots  a_s )
  \end{equation}
\end{itemize}
making $QSh(A)$ a graded coalgebra.
It is a matter of fact to check that $QSh(A)$ is a unital conilpotent bialgebra
(and thus a Hopf algebra, see e.g. \cite{Cartier2}), 
which is called the quasi-shuffle or stuffle Hopf algebra on $A$ (this terminology, that we adopt, is convenient, usual, but slightly misleading because when $A$ is only associative, $QSh(A)$ is a unital noncommutative quasi-shuffle algebra).

\section{Operations and universal properties}\label{sect:3}
Let us focus now in the first part of this section on the case relevant to renormalization, that is when $A$ is commutative but not necessarily unital. It follows then from standard arguments in universal algebra that, given a quasi-shuffle algebra $B$, morphisms of quasi-shuffle algebras from $QSh^+(A)$ to $B$ are naturally in bijection with morphisms of (non unital) algebras from $A$ to $B$:
$$Hom_{QSh}(QSh^+(A),B)\cong Hom_{Alg}(A,B).$$
In categorical terms (see \cite{fpQSh} for a direct and elementary proof):

\begin{proposition}[Quasi-shuffle PBW theorem] The left adjoint $U$ of the
forgetful functor from the category of quasi-shuffle algebras $QSh$ to the category of non unital commutative algebras $Com$, or "quasi-shuffle enveloping algebra"
functor from $Com$ to $QSh$, is  Hoffman's quasi-shuffle
algebra functor $A\longmapsto QSh^+(A)$.\end{proposition}

It is interesting to analyse the concrete meaning of this Proposition. Let us consider first the counit of the adjunction, that is the quasi-shuffle algebra map from $QSh^+(A)$ to $A$, when $A$ is a quasi-shuffle algebra.
By definition of $\prec$, the element $a_1\dots a_n\in QSh(A)_n$ can be rewritten (in $QSh(A)$) $a_1\prec (a_2\prec \dots (a_{n-1}\prec a_n))$. The trick goes back to Sch\"utzenberger who used it in his seminal but not enough acknowledged study of shuffle algebras \cite{schutz}. 
It follows that the counit of the adjunction maps $a_1\dots a_n\in QSh(A)_n$ to $a_1\prec (a_2\prec \dots (a_{n-1}\prec a_n))$ (computed now in $A$).

Let us move now to the case when $A$ is a commutative RB algebra. Then, $A$ is in particular a quasi-shuffle algebra with $a\prec b:=aR(b)$. The counit of the same adjunction is then the map that sends $a_1\dots a_n\in QSh(A)_n$ to $a_1R(a_2 R(a_3 \dots R(a_{n-1}R( a_n)))$. In particular, $a^n$ is mapped to $aR(a R(a \dots R(aR( a)))$ -a term that is known to play a key role in renormalization, see in particular \cite{emp}.

This relatively standard adjunction analysis can be completed in the case we are interested in (maps from $QSh^+(A)$ to $B$, when $B$ is a quasi-shuffle algebra), due to the existence of a Hopf algebra structure on $QSh(A)$.
According to Proposition \ref{convol}, we have first that 
\begin{lemma} Let $A$ be an associative algebra and $B$ a NQSh algebra, the vector space of linear morphisms ${\mathcal L}(
QSh(A), B )$ is a NQSh algebra.\end{lemma}
   
Furthermore, by properties that hold for arbitrary maps from a conilpotent Hopf algebra to an algebra, if $B$ is unital, the set of linear maps that map the unit of $QSh(A)$ to the unit of $B$, $\mathcal{U}( QSh(A), B )$ is a group for the product $\star$.
Moreover, when $B$ is commutative, the subset of algebra maps from $QSh(A)$ to $B$, $\mathcal{C} ( QSh(A), B )$, is a
subgroup.

Next, notice that the functor $QSh$ is compatible with Hopf algebra structures: an algebra map $l$ from $A$ to $B$  induces a map $QSh(l)$ of quasi-shuffle algebras from $QSh(A)$ to $QSh(B)$ defined by
\[ QSh(l)( 1 ) = 1 \hspace{1em} \tmop{and}
   \hspace{1em} QSh(l) ( a_1  \ldots  a_r ) = l ( a_1 )
    \ldots  l ( a_r ) \hspace{1em} ( r \geq 1 ) \]
and therefore $\Delta \circ QSh(l) = ( QSh(l)
\otimes QSh(l )) \circ \Delta$. In particular, $QSh(l)$ is a Hopf algebra morphism.

The last universal property of the $QSh$ functor that we would like to emphasize is more intriguing and does not seem to have been noticed before. Whereas $QSh$ is naturally a left adjoint, it also happens indeed to be a right adjoint, a property that will prove essential in our later developments.

\begin{theorem}\label{mainthm}
Let $H$ be a conilpotent Hopf algebra and $A$ be a unital associative algebra, then we have a natural isomorphism between (unital) algebra maps from $H$ to $A$ and Hopf algebra maps from $H$ to $QSh(A)$:
$$Hom_{Alg}(H,A)\cong Hom_{Hopf}(H,QSh(A)).$$
\end{theorem}

Indeed, $QSh(A)$ is, as a coalgebra, the cofree coalgebra over $A$ (viewed as a vector space) in the category of conilpotent coalgebras. These properties are dual to the ones of tensor algebras (more familiar, but equivalent up to the fact that the dual of a coalgebra is an algebra but the converse is not always true -this is the reason for the conilpotency hypothesis): the tensor algebra over a vector space $V$ is, when equipped with the concatenation product, the free associative algebra over $V$.
There is therefore a natural isomorphism between linear maps from  the kernel $C^+$ of the counit of a coaugmented conilpotent coalgebra $C$ to $A$ and coalgebra maps from $C$ to $QSh(A)$
$${\mathcal L}(C^+,A)\cong Hom_{Coalg}(C,QSh(A)).$$
Coaugmented means that there is a coalgebra map from the ground field to $C$, insuring that $C$ decomposes as the direct sum of $k$ and of the kernel of the counit (as happens for a Hopf algebra, for which the composition of the unit and the counit is a projection on the ground field orthogonally to the kernel of the counit). 

The isomorphism is given explicitly as follows: it maps $\phi\in {\mathcal L}(C^+,A)$ to $\tilde\phi :=\sum\limits_{i=0}^\infty \phi^{\otimes n}\circ {\Delta'}^{[n]}$ (where $\phi^{\otimes 0}\circ{\Delta'}^{[0]}$ stands for the composition of the counit of $C$ with the unit of $QSh(A)$).
In particular, the map $\phi$ factorizes as (the restriction to $C^+$ of) $j\circ \tilde \phi$, where
$j \in {\mathcal L}( QSh(A), A )$  is defined by $j (
1 ) = 1_A$, $j ( a_1 ) = a_1$ and $j ( a_1  \ldots  a_r
) = 0$ if $r \geq 2$.

To prove the Theorem, it is therefore enough to show that, when a linear map $\phi$ from $H^+$ to $A$ is the restriction to $H^+$ of an algebra map from $H$ to $A$, then the induced map $\tilde\phi$ is also an algebra map (since we already know it is a coalgebra map).
Concretely, we have to prove that, for $h,h'\in H^+$, $\tilde\phi (hh')=\tilde\phi(h)\qshuffle \tilde\phi (h')$. The Theorem will then follow if we prove that 
$$\sum\limits_{n=1}^\infty \phi^{\otimes n}\circ {\Delta'}^{[n]}(hh')=\sum\limits_{p=1}^\infty\phi^{\otimes p}\circ {\Delta'}^{[p]}(h)\qshuffle \sum\limits_{q=1}^\infty \phi^{\otimes p}\circ {\Delta'}^{[q]}(h').$$
Using that $\phi$ and that $\Delta$ are algebra maps,   this follows from the following Lemma (where, to avoid ambiguities, we use the notation ${\Delta'}^{[p]}(h)=h_{(1,p)}'\otimes \dots \otimes h_{(p,p)}'$) by identification of the terms in the left and right hand side.

\begin{lemma} We have, for the iterated coproduct and $h\in H^+$,
$$\Delta^{[n]}(h)=\sum_{i=1}^n \sum_{f\in Inj(i,n)}f_\ast (h_{(1,i)}'\otimes \dots \otimes h_{(i,i)}'),$$
where $Inj(i,n)$ stands for the set of increasing injections from $[i]:=\{1,\dots ,i\}$ to $[n]$ and 
$$f_\ast (h_{(1,i)}'\otimes \dots \otimes h_{(i,i)}')=l_{(1)}\otimes \dots \otimes l_{(n)}$$ with $l_{(q)}:=h'_{(p,i)}$ if $q=f(p)$ and $l_{(q)}:=1$ if $q$ is not in the image of $f$.
\end{lemma}

For example, $\Delta^{[1]}(h)={\Delta'}^{[1]}(h)=h=h_{(1,1)}'$, $$\Delta^{[2]}(h)=\Delta(h)=h_{(1,1)}'\otimes 1+1\otimes h_{(1,1)}'+h_{(1,2)}'\otimes  h_{(2,2)}'$$ and
$$\begin{array}{rcl}
\Delta^{[2]}(hk)&=&\Delta^{[2]}(h)\Delta^{[2]}(k) \\
&=& (h_{(1,1)}'\otimes 1+1\otimes h_{(1,1)}'+h_{(1,2)}'\otimes  h_{(2,2)}')\\
& &\ \ \times \ (k_{(1,1)}'\otimes 1+1\otimes k_{(1,1)}'+k_{(1,2)}'\otimes  k_{(2,2)}'),\end{array}$$
so that 
$${\Delta'}^{[2]}(hk)=h_{(1,1)}'\otimes k_{(1,1)}'+k_{(1,1)}'\otimes h_{(1,1)}'+h_{(1,1)}'k_{(1,2)}'\otimes  k_{(2,2)}'+
k_{(1,2)}'\otimes  h_{(1,1)}'k_{(2,2)}'$$
$$+h_{(1,2)}'k_{(1,1)}'\otimes  h_{(2,2)}'  
+h_{(1,2)}'\otimes  h_{(2,2)}'k_{(1,1)}'+h_{(1,2)}'k_{(1,2)}'\otimes  h_{(2,2)}'k_{(2,2)}',$$
where one recognizes the tensor degree $2$ component of $$({\Delta'}^{[1]}(h)+{\Delta'}^{[2]}(h))\qshuffle ({\Delta'}^{[1]}(k)+{\Delta'}^{[2]}(k)).$$

The Theorem has an important corollary, that we state also as a Theorem in view of its importance for our approach to renormalization.

\begin{theorem}
Let $H$ be a conilpotent bialgebra, then, the unit, written $\iota$, of the adjunction in the previous Theorem, ($\iota(1):=1$ and
$\forall h\in H^+, \iota(h)=\sum\limits_{k\geq 1}{\Delta'}^{[k]}(h))$
defines an injective Hopf algebra morphism from $H$ to $QSh(H^+)$.
In particular, any conilpotent (resp. conilpotent commutative) Hopf algebra embeds into a noncommutative quasi-shuffle (resp. a quasi-shuffle) Hopf algebra.
\end{theorem}

We let the reader check the following Lemma, that will be important later in the article and makes Theorem \ref{mainthm} more precise:
\begin{lemma}
The map $j \in {\mathcal L}( QSh(A), A )$ is a morphism of algebras.
\end{lemma}

\section{The map $j \in \mathcal{U}( QSh(A), A )$.}\label{sect:4}

We shall now illustrate the ideas of the previous section on the
 map $j \in \mathcal{U}( QSh(A), A )$ (recall it is defined by $j (
1) = 1_A$, $j ( a_1 ) = a_1$ and $j ( a_1  \ldots  a_r
) = 0$ if $r \geq 2$). In a sense, this will be the only computation of inverse
and of Birkhoff decomposition we will need. This map $j$ plays a fundamental role. We already saw that it appears in the adjunction ${\mathcal L}(C^+,A)\cong Hom_{Coalg}(C,QSh(A)).$ It will also appear later to be the unit of a semigroup structure on $\mathcal{U}( QSh(A), A )$ to be introduced in the next section.

For the inverse, we get  $j^{\ast - 1}$ :
\[ 
     j^{\ast - 1}  = u_A \circ \eta + \sum_{k \geq 1} ( u_A
     \circ \eta - j )^{\ast k}
 \]
Which means that $j^{\ast - 1} ( 1 ) = 1_A$ and for
$\tmmathbf{a}= a_1  \ldots  a_s \in {QSh(A)}^+$,
\begin{equation}
  \begin{array}{ccc}
    j^{\ast - 1} (\tmmathbf{a}) & = &\displaystyle \sum_{k \geq 1} ( - 1 )^k m_A^{[ k ]}
    \circ j^{\otimes_{} k} \circ {\Delta'}^{[ k ]} (\tmmathbf{a})\\
    & = &\displaystyle \sum_{k \geq 1} ( - 1 )^k \sum_{\tmmathbf{a}^1 \cdot \ldots
    \cdot \tmmathbf{a}^k =\tmmathbf{a} \atop \tmmathbf{a}^i\in {QSh(A)}^+} m_A^{[ k ]} \circ j^{\otimes k}
    (\tmmathbf{a}^1 \otimes \ldots \otimes
    \tmmathbf{a}^k )\\
    & = &\displaystyle ( - 1 )^s  m_A^{[ s ]}
    ({a}_1 \otimes \ldots \otimes
    {a}_s )\\
    & = & ( - 1 )^s a_1\cdot_A \ldots\cdot_A a_s=j\circ S(\tmmathbf{a})
  \end{array}
\end{equation}
where
\begin{equation}
  \begin{array}{lll}
    S (\tmmathbf{a}) & = & \displaystyle \sum_{k \geq 1} ( - 1 )^k m^{[ k ]}
    \circ \Delta'^{[ k ]} (\tmmathbf{a})\\
    & = & \displaystyle\sum_{k \geq 1} ( - 1 )^k \sum_{\tmmathbf{a}^1 \cdot \ldots
    \cdot \tmmathbf{a}^k =\tmmathbf{a} \atop \tmmathbf{a}^i\in {QSh(A)}^+} 
    \tmmathbf{a}^1 \qshuffle \ldots \qshuffle
    \tmmathbf{a}^k .
  \end{array}
\end{equation}
Note that the previous sums run over all the possible factorizations in nonempty words of $\tmmathbf{a}$ for the concatenation product.

If $( A, p_+ )$ is a Rota-Baxter algebra then the  Bogoliubov preparation map
$\bar{j}$ associated to $j$, see equation (\ref{eq:Bog}), is such that $\bar{j} (1 ) = 1_A$ and can be
defined recursively on vector spaces $QSh(A)_n$ ($n \geq 1$) by
\begin{equation}
  \bar{j} ( h ) = j ( h ) - m_A \circ ( p_- \otimes \tmop{Id} ) \circ (
  \bar{j} \otimes j ) \circ \Delta' ( h )
\end{equation}
Let us begin the recursion on the length of the sequence. If $\tmmathbf{a}=
a_1$ then $\bar{j} ( a_1 ) = j ( a_1 ) = a_1$. Now, if $\tmmathbf{a}=
a_1 \cdot a_2=a_1 a_2$,
\begin{equation}
  \bar{j} ( a_1  a_2 ) = j ( a_1  a_2 ) - m_A \circ ( p_-
  \otimes \tmop{Id} ) \circ ( \bar{j} \otimes j ) ( ( a_1 )
  \otimes ( a_2 ) ) = - p_- ( a_1 )\cdot_A a_2
\end{equation}
and
\begin{equation}
  \begin{array}{ccc}
    \bar{j} ( a_1 a_2 a_3 ) & = & - m_A \circ ( p_- \otimes \tmop{Id} )
    \circ ( \bar{j} \otimes j ) ( ( a_1  a_2 ) \otimes (
    a_3 ) )\\
    & = & p_- ( p_- ( a_1 )\cdot_A a_2 )\cdot_A a_3
  \end{array}
\end{equation}
Thus, for $r \geq 2$,
\begin{equation}
  \bar{j} ( a_1  \ldots  a_r ) = - p_- ( \bar{j} ( a_1 \ldots
  a_{r - 1} ) )\cdot_A a_r
\end{equation}
It is then easy to prove that in general (see e.g. \cite{emp} for a systematic study of combinatorial approaches and closed solutions  to the Bogoliubov recursion)

\begin{proposition}
  The Birkhoff decomposition 
  $$( j_+, j_- ) \in \mathcal{U}(
  QSh(A), A_+ ) \times \mathcal{U}( QSh(A), A_- )$$ such that
  \[ j_- \ast j = j_+ \]
  is given by the formula : for $r \geq 1$ and $\tmmathbf{a}= a_1 \otimes
  \ldots \otimes a_r \in {QSh(A)}^+$,
  \begin{equation}\label{eq:Birkj}
    \hspace{1em} \left\{\begin{array}{lllll}
      j_+ (\tmmathbf{a}) & = & p_+ ( \bar{j} (\tmmathbf{a}) ) & = & ( - 1 )^{r
      - 1} p_+ ( p_- ( \ldots ( p_- ( a_1 )\cdot_A a_2 ) \ldots \cdot_A a_{r - 1} )\cdot_A a_r )\\
      j_- (\tmmathbf{a}) & = & - p_- ( \bar{j} (\tmmathbf{a}) ) & = & ( - 1
      )^r p_- ( p_- ( \ldots ( p_- ( a_1 )\cdot_A a_2 ) \ldots \cdot_Aa_{r - 1} ) \cdot_A a_r )
    \end{array}\right.
  \end{equation}
  Moreover, if $A$ is commutative then $\mathcal{C}( QSh(A), A )$ is
  a group and $j_{_+}$ and $j_-$ are characters.
\end{proposition}

\begin{proof}
  Let us prove the last assumption, when $A$ is commutative. Since $j$
  is a character it is sufficient to prove that $j_-$ is a character. By
  induction on $t \geq 0$ we will show that for two tensors
  $\tmmathbf{a}$ and $\tmmathbf{b}$ in $QSh(A)$, if $l (\tmmathbf{a}) +
  l (\tmmathbf{b}) = t$, then
  \begin{equation}
    j_-   (\tmmathbf{a} \qshuffle
    \text{$\tmmathbf{b}$}  ) = j_- (\tmmathbf{a}) j_- (\tmmathbf{b})
  \end{equation}
  This identity is trivial for $t = 0$ and $t = 1$ since at least one of the
  sequences is the empty sequence. This also trivial for any $t$ if one of the
  sequences is empty. Now suppose that $t \geq 2$ and that $\tmmathbf{a}= a_1
   \ldots  a_r \in QSh(A)_{ r }$ and $\tmmathbf{b} =
  b_1  \ldots  b_s \in QSh(A)_{s }$ with $r \geq 1$,
  $s \geq 1$ and $r + s = t$. 
  Let $\tilde{\tmmathbf{a}} = a_1 \ldots
   a_{r - 1} \in QSh(A)_{r - 1 }$ ($\tilde{\tmmathbf{a}} =
  1$ if $r = 1$) and $\tmmathbf{\tilde{b}} = b_1  \ldots
   b_{s - 1} \in QSh(A)_{s - 1 }$ ($\tmmathbf{\tilde{b}} =
  1$ if $s = 1$), then :
  \[ \tmmathbf{a} \qshuffle \tmmathbf{b}  =
    ( \tilde{\tmmathbf{a}} \qshuffle \tmmathbf{b} )
     \cdot a_r +  (\tmmathbf{a} \qshuffle
     \tmmathbf{\tilde{b}} ) \cdot b_s +  ( \tilde{\tmmathbf{a}}
     \qshuffle\tmmathbf{\tilde{b}} ) \cdot  (a_r \cdot_A b_s) \]
  Now we have
  \[ j_- (\tmmathbf{a}) = - p_- ( j_- ( \tilde{\tmmathbf{a}} ) \cdot_A a_r ) = - p_- (
     x ) \hspace{1em} \tmop{and} \hspace{1em} j_- (\tmmathbf{b}) = - p_- ( j_-
     ( \tmmathbf{\tilde{b}} ) \cdot_A b_s ) = - p_- ( y ), \]
     where $x:=j_- ( \tilde{\tmmathbf{a}} ) \cdot_A a_r$ and $y:=j_-
     ( \tmmathbf{\tilde{b}} ) \cdot_A b_s$.
  Thanks to the Rota-Baxter identity, and omitting $\cdot_A$ in the following computations in $A$,
  \[ \begin{array}{ccc}
       j_- (\tmmathbf{a}) j_- (\tmmathbf{b}) & = & p_- ( x ) p_- ( y )\\
       & = & p_- ( x p_- ( y ) ) + p_- ( p_- ( x ) y ) - p_- ( x y )\\
       & = & p_- ( j_- ( \tilde{\tmmathbf{a}} ) a_r p_- ( j_- (
       \tmmathbf{\tilde{b}} ) b_s ) ) + p_- ( p_- ( j_- ( \tilde{\tmmathbf{a}}
       ) a_r ) j_- ( \tmmathbf{\tilde{b}} ) b_s ) \\
        & & \quad - p_- ( j_- (
       \tilde{\tmmathbf{a}} ) a_r j_- ( \tmmathbf{\tilde{b}} ) b_s )
     \end{array} \]
  but as $A$ is commutative, by induction we get
  \[ \begin{array}{ccc}
       j_- (\tmmathbf{a}) j_- (\tmmathbf{b}) & = & - p_- ( j_- (
       \tilde{\tmmathbf{a}} ) j_- (\tmmathbf{b}) a_r ) - p_- ( j_-
       (\tmmathbf{a}) j_- ( \tmmathbf{\tilde{b}} ) b_s ) - p_- ( j_- (
       \tilde{\tmmathbf{a}} ) j_- ( \tmmathbf{\tilde{b}} ) a_r b_s )\\
       & = & - p_- ( j_-   ( \tilde{\tmmathbf{a}} \qshuffle
       \tmmathbf{b}  ) a_r ) - p_- ( j_-   (\tmmathbf{a}
       \qshuffle \tmmathbf{\tilde{b}}  ) b_s ) - p_- ( j_- ( 
       \tilde{\tmmathbf{a}} \qshuffle \tmmathbf{\tilde{b}}  ) a_r b_s )\\
       & = & j_- (  ( \tilde{\tmmathbf{a}} \qshuffle \tmmathbf{b}
       ) \cdot a_r ) + j_- (  (\tmmathbf{a} \qshuffle
       \tmmathbf{\tilde{b}} ) \cdot b_s ) + j_- (  (
       \tilde{\tmmathbf{a}} \qshuffle \tmmathbf{\tilde{b}} ) \cdot (a_r  b_s ))\\
       & = & j_- (  ( \tilde{\tmmathbf{a}} \qshuffle \tmmathbf{b}
       ) \cdot a_r +  (\tmmathbf{a} \qshuffle
       \tmmathbf{\tilde{b}} ) \cdot b_s +  (
       \tilde{\tmmathbf{a}} \qshuffle \tmmathbf{\tilde{b}}  ) \cdot
       (a_r b_s) )\\
       & = & j_- (  \tmmathbf{a} \qshuffle \tmmathbf{b}
       ) 
     \end{array} \]
  
\end{proof}

In the sequel, when there is no ambiguity, we shall omit the notation $\cdot_A$ when applying formula (\ref{eq:Birkj}).

As we will see these formulas are almost sufficient to compute the Birkhoff
decomposition for any conilpotent bialgebra.

\section{The universal semigroup and renormalization.} \label{sect:5}

Let $A$ be a unital algebra. Then, by adjunction we know that $$\mathcal{U}(
QSh(A), A )\cong Hom_{Coalg}(
QSh(A), QSh(A) ).$$ In particular, the composition of coalgebra endomorphisms of $QSh(A)$ equips $\mathcal{U}(
QSh(A), A )$ with a semigroup structure.

\begin{definition}
The universal semigroup associated to a unital algebra $A$ is the set $\mathcal{U}(
QSh(A), A )$ equipped with the associative unital product induced by composition of coalgebra endomorphisms of $QSh(A)$:
for $f$ and $g$ in $\mathcal{U}(
QSh(A), A )$ 
\[ f \odot g := f\circ QSh(g)\circ \iota . \]
Its unit is the map $j$:
\[ f \odot j = f \circ QSh(j) \circ \iota = f\circ Id=f . \]
\end{definition}

This semigroup structure generalizes to an action on linear maps from a Hopf algebra to $A$ as follows.

\begin{definition}
Let $H$ be a conilpotent bialgebra. For
$\varphi \in \mathcal{U}( H, A )$ and $f \in \mathcal{U}( QSh(A), A )$
we set
\[ f\odot \varphi : = f \circ QSh(\varphi) \circ \iota .\]
This morphism $f\odot \varphi $ is  linear  from $H$ to $A$ and unital:
\[ f\odot \varphi  ( 1_H ) = f \circ QSh(\varphi) \circ \iota ( 1_H )
   = f \circ QSh(\varphi) ( 1 ) = f ( 1 ) = 1_A .\]
We get a left action of $\mathcal{U}( QSh(A), A )$ on $\mathcal{U}( H, A )$:
$$\odot : \mathcal{U}( QSh(A), A )\times \mathcal{U}( H, A )\to \mathcal{U}( H, A ).$$
Moreover, when $A$ is commutative, if $\varphi \in
\mathcal{C}( H, A )$ and  $f \in \mathcal{C}( QSh(A), A )$ it is clear,
by composition of algebra morphisms, that $ f\odot \varphi  \in \mathcal{C}( H,
A )$.
\end{definition}

That $j$ acts as the identity map on $\mathcal{U}( H, A )$ follows from: for $h \in H^+$,
  \begin{equation}
    \begin{array}{ccc}
      j\odot \varphi  ( h ) & = &\displaystyle j \circ QSh(\varphi)\left( h +
      \sum_{k \geq 2} \sum h'_{( 1 )} \otimes \ldots \otimes h'_{( k )}
      \right)\\
      & = &\displaystyle j \left( \varphi ( h ) + \sum_{k \geq 2} \varphi ( h'_{( 1 )} )
      \cdot \ldots \cdot \varphi ( h'_{( k )} ) \right)\\
      & = & \varphi ( h )
    \end{array}
  \end{equation}

\begin{proposition} The action $\odot$ and the convolution product $\ast$ (recall that $QSh(A)$ is a Hopf algebra) satisfy the distributivity relation:
For $f$ and $g$ in $\mathcal{U}( QSh(A), A )$ and $\varphi$ in
  $\mathcal{U}( H, A )$, 
  $$(f\ast g)\odot \varphi=(f\odot \varphi)\ast (g\odot \varphi).$$
  \end{proposition}
 Indeed, \begin{equation}
    \begin{array}{ccc}
      (f\ast g)\odot \varphi & = & m_A \circ ( f \otimes g ) \circ
      \Delta \circ QSh(\varphi) \circ \iota\\
      & = & m_A \circ ( f \otimes g ) \circ ( QSh(\varphi) \otimes
      QSh(\varphi) ) \circ \Delta \circ \iota\\
      & = & m_A \circ ( f \otimes g ) \circ ( QSh(\varphi) \otimes
      QSh(\varphi) ) \circ ( \iota \otimes \iota ) \circ \Delta\\
      & = & m_A (  f\odot \varphi  \otimes g\odot \varphi  ) \circ \Delta\\
     & = & (f\odot \varphi ) \ast  ( g\odot \varphi )
    \end{array}
  \end{equation}  
  
Note that, in the case $H=QSh(A)$, $\mathcal{U}( QSh(A), A )$ is equipped with two products $\ast$ and $\odot$ that look similar, in their interactions, to the product and composition of power series.

\begin{remark}
These constructions generalize as follows. Let $B$ be another unital algebra. For
$\varphi \in \mathcal{U}( H, A )$ and $f \in \mathcal{U}( QSh(A), B )$
we define
\[ f\odot \varphi = f \circ QSh(\varphi) \circ \iota .\]
The morphism $f\odot\varphi $ is linear  from $H$ to $B$ and
\[  f\odot \varphi  ( 1_H ) = f \circ QSh(\varphi)\circ \iota ( 1_H )
   = f \circ QSh(\varphi) ( 1 ) = f ( 1 ) = 1_B .\]
thus $f\odot \varphi  \in \mathcal{U}( H, B )$. Moreover, when $A$ and $B$ are commutative, if $\varphi \in
\mathcal{C}( H, A )$ and  $f \in \mathcal{C}( QSh(A), B )$ it is clear,
by composition of algebra morphisms that $ f\odot \varphi  \in \mathcal{C}( H,
B )$.
\end{remark}

\begin{corollary}
Let $\varphi \in \mathcal{U}( H, A )$, then its convolution inverse if given by
$$\varphi^{\ast -1}=j^{\ast -1}\odot \varphi.$$
\end{corollary}

Indeed, since $j\odot \varphi  = \varphi$, if
$\psi := j^{\ast - 1} \odot\varphi $, then
\[ \psi \ast \varphi =  ( j^{\ast - 1}\odot \varphi ) \ast ( j\odot \varphi ) = (
   j^{\ast - 1} \ast j)\odot \varphi  =  ( u_A \circ \eta)\odot \varphi 
   = u_A \circ \eta .\]

For example, if $h \in H^+$ with $\Delta'^{[4]}(h)=0$, then
\[ \iota ( h ) = h + \sum h'_{( 1 )} \otimes h'_{( 2 )} + \sum h'_{( 1 )} \otimes h'_{(
   2 )} \otimes h'_{( 3 )} \]
so,
\[ QSh(\varphi) \circ \iota ( h ) = \varphi ( h ) + \sum\varphi ( h'_{( 1
   )} ) \cdot \varphi ( h'_{( 2 )} ) +\sum \varphi ( h'_{( 1 )} ) \cdot
   \varphi ( h'_{( 2 )} ) \cdot \varphi ( h'_{( 3 )} ) \]
and finally
\[ \varphi^{\ast - 1} ( h ) = j^{\ast - 1} \circ QSh(\varphi) \circ
   \iota(h) = - \varphi ( h ) +\sum \varphi ( h'_{( 1 )} ) \varphi ( h'_{( 2 )} ) -
  \sum \varphi ( h'_{( 1 )} ) \varphi ( h'_{( 2 )} ) \varphi ( h'_{( 3 )} ) \]
We recover the usual formula for the inverse.

\begin{theorem}Assume now that $A$ is an idempotent Rota--Baxter algebra.
Let $\varphi \in \mathcal{U}( H, A )$. Then the Birkhoff-Rota-Baxter decomposition of $\varphi$ is given by
$$\varphi_-=j_-\odot \varphi,\ \ \varphi_+=j_+\odot\varphi .$$
\end{theorem}

\begin{proof} Indeed, since  $ j\odot \varphi  =
\varphi$, we have
\[ \varphi_- \ast \varphi =  ( j_-\odot \varphi ) \ast  ( j\odot \varphi ) =  ( j_-
   \ast j)\odot \varphi  =  j_+\odot \varphi  = \varphi_+ \]
and, of course, $\varphi_{\pm} \in \mathcal{U}( H, A_{\pm} )$. 
\end{proof}

For example,
if $h \in H'$ with $\Delta'^{[4]}(h)=0$, then
\[ \begin{array}{ccc}
     \varphi_+ ( h ) & = & p_+ ( \varphi ( h ) ) -\sum p_+ ( p_- ( \varphi ( h'_{(
     1 )} ) ) \varphi ( h'_{( 2 )} ) ) + \sum p_+ ( p_- ( p_- ( \varphi ( h'_{( 1
     )} ) ) \varphi ( h'_{( 2 )} ) ) \varphi ( h'_{( 3 )} ) )\\
     & & \\
     \varphi_- ( h ) & = & - p_- ( \varphi ( h ) ) +\sum p_- ( p_- ( \varphi (
     h'_{( 1 )} ) ) \varphi ( h'_{( 2 )} ) ) -\sum p_- ( p_- ( p_- ( \varphi (
     h'_{( 1 )} ) ) \varphi ( h'_{( 2 )} ) ) \varphi ( h'_{( 3 )} ) )
   \end{array} \]

Needless to say that if $A$ is commutative, these computations works in the
subgroup $\mathcal{C}( H, A )$.

Once these formulas are given, we get formulas in the different contexts where
renormalization, or rather Birkhoff decomposition, is needed. We end this paper with two sections that illustrate how these formulas could be used :
\begin{itemize}
\item to perform inversion and Birkhoff decomposition of diffeomorphisms that correspond to characters on the Faà di Bruno Hopf algebra,
\item  to perform the Birkhoff decomposition with the same formula in various cofree Hopf algebras that differ by their algebra structures, but for which the map $\iota$ is the same as these Hopf algebras are tensor coalgebras.
\end{itemize}

\section{Renormalizing diffeomorphisms in pQFT and Dynamics}\label{sect:6}

Let us focus in this section on the example of the  Faà di Bruno Hopf algebra $\mathcal{H}_{\tmop{FdB}}$ (see \cite{BF,fig,MF,men}) 
whose group of characters corresponds to the group of formal identity-tangent diffeomorphisms. We will first express the reduced coproduct and then the map $\iota$ 
from this Hopf algebra to its associated quasi-shuffle Hopf algebra and then focus on the Birkhoff  decomposition of characters with values 
in the Laurent series that appear in several areas, as a factorisation of diffeomorphisms for the composition.

Recall that the decomposition is unique: the same results could be obtained by induction using the classical renormalization process (the Bogoliubov recursion). One advantage of the present approach is to encode the combinatorics of renormalization into a universal framework, probably similar to the one P. Cartier suggested when advocating the existence of a ``Galois group'' underlying renormalization.
Compare in particular our approach with \cite{ck1,guo,emp}.

Consider the group of formal identity tangent diffeomorphisms with coefficients in a commutative $\mathbb{C}$--algebra $A$:
\[ G(A) = \{ f ( x ) = x + \sum_{n \geq 1} f_n x^{n + 1} \in A[ [ x ]
   ] \} \]
   with its product $\mu : G(A) \times G(A) \rightarrow G(A)$ :
\[ \mu ( f, g ) = f \circ g. \]
For $n \geq 0$, the functionals on $G(A)$ defined by
\[ a_n ( f ) = \frac{1}{( n + 1 ) !_{}} ( \partial_x^{n + 1} f ) ( 0 ) = f_n
   \hspace{1em} a_n : G(A) \rightarrow A \]
are called de Faà di Bruno coordinates on the group $G(A)$ and $a_0 = 1$ being
the unit, they generates a graded unital commutative algebra
\[ \mathcal{H}_{\tmop{FdB}} =\mathbbm{C}[ a_1, \ldots, a_n, \ldots ]
   \hspace{1em} ( \tmop{gr} ( a_n ) = n ) \]
The action of these functionals on a product in $G(A)$ defines a
coproduct on $\mathcal{H}_{\tmop{FdB}}$ that turns to be a graded connected
Hopf algebra (see {\cite{fig}} for details). For $n \geq 0$, the coproduct is
defined by
\begin{equation}
  a_n \circ \mu = m \circ \Delta ( a_n ) \label{copfdb}
\end{equation}
where $m$ is the usual product in $A$, and the antipode reads
\[ S \circ a_n = a_n \circ \tmop{inv} \]
where $\tmop{inv} ( \varphi ) = \varphi^{\circ {- 1}}$ is the composition inverse of
$\varphi$.

For example if $f ( x ) = x + \sum_{n \geq 1} f_n x^{n + 1}$ and $g ( x ) = x
+ \sum_{n \geq 1} g_n x^{n + 1}$ then if $h(x)=f \circ g(x)= x + \sum_{n \geq 1} h_n x^{n + 1}$,
\[ \begin{array}{ccccccc}
     a_0 ( h ) & = & 1 = a_0 ( f ) a_0 ( g ) & \rightarrow & \Delta a_0 & = &
     a_0 \otimes a_0\\
     a_1 ( h ) & = & f_1 + h_1 & \rightarrow & \Delta a_1 & = & a_1 \otimes
     a_0 + a_0 \otimes a_1\\
     a_2 ( h ) & = & f_2 + 2f_1 g_1 + g_2 & \rightarrow & \Delta a_2 & = & a_2
     \otimes a_0 + 2a_1 \otimes a_1 + a_0 \otimes a_2.
   \end{array} \]
   
More generally, using classical formulas on the composition of diffeomorphisms (see \cite{BF,EFP1,MF,men11}), we have
\begin{equation}
\Delta (a_n)=\sum_{k=0}^{n} \sum_{l_0+\dots l_k=n-k\atop l_i\geq 0} a_k \otimes a_{l_0}\dots a_{l_k}
\end{equation}

Let us consider sequences of positive integers
\[ \mathcal{N}= \{ \tmmathbf{n}= ( n_1, \ldots, n_s ) \in (\mathbbm{N}^{\ast}
   )^s, \hspace{1em} s \geq 1 \} \]
For $\tmmathbf{n}= ( n_1, \ldots, n_s ) \in \mathcal{N}$, we denote
\[ \| \tmmathbf{n} \| = n_1 + \ldots + n_s, \hspace{1em} l ( \tmmathbf{n} ) =
   s, \hspace{1em} a_{\tmmathbf{n}}=a_{n_1}\dots a_{n_s} \]
and, if $n \geq 1$,
\[ \mathcal{N}_n = \{ \tmmathbf{n} \in \mathcal{N} \hspace{1em} ; \hspace{1em}
   \| \tmmathbf{n} \| = n \} \]

With these notations, the reduced coproduct (with $a_0=1$) reads
\begin{equation}
\Delta'(a_n)=\sum_{k=1}^{n-1} \sum_{\tmmathbf{n} \in \mathcal{N}_{n-k}} \left( \begin{array}{c}
k+1 \\
l(\tmmathbf{n})
\end{array} \right)a_k \otimes a_{\tmmathbf{n}}
\end{equation}
and when iterating the coproduct, we get, 

\begin{proposition}\label{prop:iotafdb}
For $n\geq 1$,
\begin{equation}\label{eq:iotafdb}
\iota(a_n)=\sum_{\tmmathbf{n}\in \mathcal{N}_n}\sum_{\tmmathbf{n}^1\dots \tmmathbf{n}^t=\tmmathbf{n}\atop
t\geq 1, l(\tmmathbf{n}^1)=1} \lambda(\tmmathbf{n}^1,\dots, \tmmathbf{n}^t) a_{\tmmathbf{n}^1}\otimes \cdots \otimes a_{\tmmathbf{n}^t}
\end{equation}
where the sums run over all the decompositions in non empty sequences $\tmmathbf{n}^1\dots \tmmathbf{n}^t=\tmmathbf{n}$ and \[
\lambda(\tmmathbf{n}^1,\dots, \tmmathbf{n}^t)=\prod_{i=2}^t \left( \begin{array}{c}
\|\tmmathbf{n}^1\dots \tmmathbf{n}^{i-1}\| +1 \\
l(\tmmathbf{n}^i) \end{array} \right)
\]
\end{proposition}
Note that we kept in formula (\ref{eq:iotafdb}) the tensor product notation to avoid confusion since we deal with words whose letters are monomials. The proof is simply based on the recursive definition of reduced iterated coproduct and already provides a formula for the composition inverse of a diffeomorphism in $G(A)$. 

\begin{corollary}
 Let $f ( x ) = x + \sum_{n \geq 1} f_n x^{n + 1}\in G(A)$, we can consider its associated character defined by $\varphi(a_n)=f_n$ and then, using our previous formulas, the coefficients of the composition inverse $g$ of $f$ are given by
\[
g_n=\varphi^{*{-1}}(a_n)=\sum_{\tmmathbf{n}=( n_1, \ldots, n_s )\in \mathcal{N}_n}\left( \sum_{\tmmathbf{n}^1\dots \tmmathbf{n}^t=\tmmathbf{n}\atop
t\geq 1, l(\tmmathbf{n}^1)=1} (-1)^t \lambda(\tmmathbf{n}^1,\dots, \tmmathbf{n}^t)\right) f_{n_1}\dots f_{n_s}
\]

\end{corollary}

This result, as the following one, uses the obvious isomorphism between $G(A)$ and $\mathcal{C}(\mathcal{H}_{\tmop{FdB}},A)$. One can also compute the Birkhoff decomposition in the group of formal identity-tangent diffeomorphism with coefficients in the a Rota-Baxter algebra of Laurent series 
$A =\mathbbm{C}[ [ \varepsilon, \varepsilon^{- 1} ]$ with its usual projections $p_+$ and $p_-$ on the regular and polar parts of such series. Any element
  \[ f ( x ) = x + \sum_{n \geq 1} f_n ( \varepsilon ) x^{n+1} \hspace{1em},
     \hspace{1em} f_n ( \varepsilon ) \in \mathbbm{C}[ [ \varepsilon,  \varepsilon^{- 1} ] \]
  can be decomposed as $f_- \circ f = f_+$ with
  \[ \begin{array}{ccccc}
       f_- ( x ) & = & \displaystyle x + \sum_{n \geq 1} f_{-, n} ( \varepsilon ) x^{n+1} &  &
       f_{-, n} ( \varepsilon ) \in \varepsilon^{- 1} \mathbbm{C}[
       \varepsilon^{- 1} ]\\
       f_+ ( x ) & = & \displaystyle x + \sum_{n \geq 1} f_{+, n} ( \varepsilon ) x^{n+1} &  &
       f_{+, n} ( \varepsilon ) \in \mathbbm{C}[ [ \varepsilon ] ].
     \end{array} \]
Using  proposition \ref{prop:iotafdb}, we get for $n\geq 1$,
\begin{proposition}
The coefficients of the Birkhoff decomposition of a formal identity-tangent diffeomorphism are given by
\begin{equation}\label{bddiff}
\begin{array}{rcl}
\varphi_+(a_n) &=& \dsp \sum_{\tmmathbf{n}\in \mathcal{N}_n}\sum_{\tmmathbf{n}^1\dots \tmmathbf{n}^t=\tmmathbf{n}\atop
 t\geq 1, l(\tmmathbf{n}^1)=1} \lambda(\tmmathbf{n}^1,\dots, \tmmathbf{n}^t)(-1)^{t-1} p_+ ( p_- ( \ldots ( p_- ( \varphi(a_{\tmmathbf{n}^1}) ) \varphi(a_{\tmmathbf{n}^2}) ) \ldots  )\varphi(a_{\tmmathbf{n}^t}) )\\
 
\varphi_-(a_n) &=& \dsp \sum_{\tmmathbf{n}\in \mathcal{N}_n}\sum_{\tmmathbf{n}^1\dots \tmmathbf{n}^t=\tmmathbf{n}\atop
t\geq 1, l(\tmmathbf{n}^1)=1} \lambda(\tmmathbf{n}^1,\dots, \tmmathbf{n}^t)(-1)^{t} p_- ( p_- ( \ldots ( p_- ( \varphi(a_{\tmmathbf{n}^1}) ) \varphi(a_{\tmmathbf{n}^2}) ) \ldots  )\varphi(a_{\tmmathbf{n}^t}) )
\end{array}
\end{equation}
where $\varphi$, $\varphi_+$ and $\varphi_-$ are the characters associated to $f$, $f_+$ and $f_-$ ($\varphi(a_n)=f_n$).

\end{proposition}

Let us explain how such diffeomorphisms appear in various area, where there Birkhoff decomposition makes sense.

Such a factorization appears first classicaly in quantum field theory: after
dimensional regularization, the unrenormalized effective coupling constants
are the image by a formal identity-tangent diffeomorphism of the coupling
constants of the theory (see \cite{ck2,EFP1} for a Hopf algebraic approach). Moreover, the coefficients of this diffeomorphism are
Laurent series in the parameter $\varepsilon$ associated to the dimensional
regularization process and the Birkhoff decomposition of this diffeomorphism gives
directly the bare coupling constants and the renormalized coupling
constants. 

As proved in \cite{ck2}, in the case of the massless
$\phi^3_6$ theory, the unrenormalized effective coupling constant can be written as a diffeomorphism $f ( x ) = x + \sum_{n \geq 1} f_n ( \varepsilon ) x^{n+1}$ where $x$ is the initial coupling constant. From the physical point of view, the decomposition $f_- \circ f = f_+$ is such that,  $x + \sum_{n \geq 1} f_{+, n} ( 0) x^{n+1}$ is the renormalized effective constant of the theory.

Such diffeomorphisms (and the need for renormalization) also appear in the classification of dynamical systems, especially when dealing with dynamical systems that cannot be analytically of formally linearized. Let us illustrate this on a very simple example (see \cite{men13} for a general approach). The following autonomous analytic dynamical system 
\[
\left\lbrace
\begin{array}{rcl}
\dsp \dot{x} &=&\alpha x\\
\dsp \dot{z} &=&\beta z +b(x)z^2
\end{array}\right. 
\]
can be considered as a perturbation of the linear system
\[
\left\lbrace
\begin{array}{rcl}
\dsp \dot{x} &=&\alpha x\\
\dsp \dot{y} &=&\beta y
\end{array} \right.
\]
so that one could expect that a change of coordinate $(x,y)=\psi(x,z)=\left(x,f(x,z) \right)$ allows to go from one system to the other one, 
that is to linearize the first system. In this simple case (see \cite{men13} for details) the solution should be $f(x,z)=\frac{z}{1-a(x)z}$ where
\[
\alpha xa'(x)+\beta a(x)+b(x)=0
\]
that yields formally, if $b(x)=\sum_{n\geq 0}b_n x^n$,
 \[ a(x)=-\sum_{n\geq 0}\frac{b_n}{\alpha n + \beta} x^n. \]
This series could be ill-defined whenever there exists $n_0$ such that $ \alpha n_0 + \beta=0$. 
This happens for example with $n=0$ for $(\alpha,\beta)=(1,0)$ and, in this case, we could regularize by considering the system with linear part $(\alpha,\beta)=(1+\varepsilon,\varepsilon)$.
As a function of $z$, $f(x,z)$ is then an identity-tangent diffeomorphism whose coefficients are in $\mathbbm{C}[ [x ]][[ \varepsilon ,\varepsilon^{- 1} ]$:
\[
 f(x,z)=\frac{z}{1-a(x)z}=z+\sum_{n\geq 1}a(x)^n z^{n+1} ,\quad a(x)=-\frac{b(0)}{\varepsilon}-\sum_{n\geq1}\frac{b_n}{n(1+\varepsilon)+\varepsilon}x^n.
 \]
This very simple case can be handled directly and, after Birkhoff decomposition, the regular part in $\varepsilon$ is  
\[
 f_+(x,z)=\frac{z}{1-a_+(x)z}=z+\sum_{n\geq 1}a_+(x)^n z^{n+1} ,\quad a_+(x)=-\sum_{n\geq1}\frac{b_n}{n(1+\varepsilon)+\varepsilon}x^n
 \]
 and, for $\varepsilon=0$, the corresponding change of coordinate conjugates the system 
 \[
\left\lbrace
\begin{array}{rcl}
\dsp \dot{x} &=& x\\
\dsp \dot{z} &=& b(x)z^2
\end{array}\right. 
\]
to 
\[
\left\lbrace
\begin{array}{rcl}
\dsp \dot{x} &=& x\\
\dsp \dot{y} &=&b(0)y^2
\end{array} \right. .
\]
This approach can be generalized to more general systems for which the Birkhoff decomposition is not so obvious, so that  formula (\ref{bddiff}) could be useful. For instance, the same process of regularization/factorization allows to conjugate the system 
\[
\left\lbrace
\begin{array}{rcl}
\dsp \dot{x} &=& x\\
\dsp \dot{z} &=& \sum_{k\geq 1} b_k(x)z^{k+1}
\end{array}\right. 
\] 
to a system
\[
\left\lbrace
\begin{array}{rcl}
\dsp \dot{x} &=& x\\
\dsp \dot{y} &=& \sum_{k\geq 1} c_k y^{k+1}
\end{array}\right. 
\]
which is called a "normal form", with coefficients $c_k$ that do not depend any more on $x$.

Diffeomorphisms in higher dimension (and thus the corresponding Hopf algebra)  appear as well in physics (with more than one coupling constant) and in dynamics: 
let us consider vector fields given by $\nu$ series $\tmmathbf{u}(\tmmathbf{x})=(u_1(\tmmathbf{x}),\dots,u_{\nu}(\tmmathbf{x}))\in \CC_{\geq 2}\lbrace \tmmathbf{x}\rbrace$ 
of $\nu$ variables $\tmmathbf{x}=(x_1,\dots,x_{\nu})$ that can be seen as "perturbations" of linear vector fields $(\lambda_1 x_1,\dots,\lambda_{\nu}x_{\nu})$:
\begin{equation}
\label{eq:vf}
\frac{dx_i}{dt}=\lambda_i x_i + u_i(\tmmathbf{x})=X_i(\tmmathbf{x}),\quad i=1,\dots, \nu.
\end{equation}
The linearization problem consists in finding an identity-tangent diffeomorphism $\varphi$ in dimension $\nu$ such that the change of coordinates $\tmmathbf{x}=\varphi(\tmmathbf{y})$ transforms the previous object into its linear part.
For differential equations, this reads, for $i=1,\dots, \nu$:
\begin{equation}
\frac{dx_i}{dt}=\sum_{j=1}^{\nu} \frac{dy_j}{dt}\frac{\partial \varphi_i}{\partial y_j}(\tmmathbf{y})=\sum_{j=1}^{\nu} \lambda_j y_j \frac{\partial \varphi_i}{\partial y_j}(\tmmathbf{y})=\lambda_i \varphi_i(\tmmathbf{y})+u_i(\varphi(\tmmathbf{y}))=\lambda_i x_i + u_i(\tmmathbf{x}).\label{eq:homVF}
\end{equation}
When trying to solve these so-called "homological equations", some obstructions can occur, independently on any assumption on the analycity of $\varphi$. These equations cannot be formally systematically solved when some combinations $ m_1 \lambda_1 + \ldots m_{\nu} \lambda_{\nu} - \lambda_i$ vanish (here $ i \in \{ 1, \ldots, \nu \},\ m_j\geq0,\ \sum m_j \geqslant 2$):

Such cancellations, which are called \textit{resonances}, prevent from  linearizing the differential and one can once again use regularization of the linear part 
and Birkhoff decomposition to get a change of coordinate that conjugate the vector field to a so-called normal form, see \cite{men13}.

\section{Tensor coalgebras, MZVs, Analysis}

If $X$ be an alphabet (that is a set), its associated tensor vector space $T(X)$ inherits a coalgebra structure related to the concatenation. If we note tensors products as words $\tmmathbf{x}=x_1\otimes \dots\otimes x_s=x_1\dots x_s$,
\[
\Delta(\tmmathbf{x})=1\otimes \tmmathbf{x} +\sum_{\tmmathbf{x}^1\tmmathbf{x}^2=\tmmathbf{x}} \tmmathbf{x}^1\otimes \tmmathbf{x}^2 + \tmmathbf{x} \otimes 1
\]
where the central sum, that corresponds to the reduced coproduct, is over nonempty words $\tmmathbf{x}^1,\tmmathbf{x}^2$ whose concatenation is $\tmmathbf{x}$.

The quasi-shuffle Hopf algebras $QSh(A)$ are examples of such coalgebras (choose simply a linear basis $X$ of $A$!).
There are however many Hopf algebras with such a coalgebra structure that differ as algebras --but the associated map $\iota$ and the associated formula for the Birkhoff decomposition of characters, does not depend on the algebra structure. For the map $\iota$, we obviously get:
\begin{equation}
\iota(\tmmathbf{x})= \sum_{\tmmathbf{x}^1\tmmathbf{x}^2\dots \tmmathbf{x}^t=\tmmathbf{x}\atop t\geq 1\ ;\ \tmmathbf{x}^i\not= \emptyset} \tmmathbf{x}^1\otimes \tmmathbf{x}^2 \otimes \dots \otimes \tmmathbf{x}^t
\end{equation}
and if $\varphi$ is a character from a Hopf algebra with such a coalgebra structure, with values in a commutative Rota-Baxter algebra $(A,p_+)$, the factorization $\varphi_- \ast \varphi=\varphi_+$ is given for any $\tmmathbf{x}\in T(X)$ by

\begin{equation} \label{eq:wordbirk}
\begin{array}{ccc}
     \varphi_+ ( \tmmathbf{x} ) & = & \dsp \sum_{\tmmathbf{x}^1\tmmathbf{x}^2\dots \tmmathbf{x}^t=\tmmathbf{x}\atop t\geq 1\ ;\ \tmmathbf{x}^i\not= \emptyset} (-1)^{t-1} p_+ ( p_- ( \ldots ( p_- ( \varphi(\tmmathbf{x}^1)) \varphi( \tmmathbf{x}^2) ) \ldots ) \varphi(\tmmathbf{x}^t)) \\
     \varphi_- ( \tmmathbf{x} ) & = & \dsp \sum_{\tmmathbf{x}^1\tmmathbf{x}^2\dots \tmmathbf{x}^t=\tmmathbf{x}\atop t\geq 1\ ;\ \tmmathbf{x}^i\not= \emptyset} (-1)^{t} p_- ( p_- ( \ldots ( p_- ( \varphi(\tmmathbf{x}^1)) \varphi( \tmmathbf{x}^2) ) \ldots ) \varphi(\tmmathbf{x}^t))
   \end{array} 
   \end{equation}
   
Let us list some example where this formula appear or can be used.
\begin{example}[Renormalization af Multiple Zeta Values (MZV)]
In \cite[Section 3]{GZ} Guo and Zhang consider regularized MZV as characters on a quasi-shuffle algebra $\mathcal{H}_{\mathfrak{M}}=T(\mathfrak{M})$ whose quasi-shuffle product stems from the additive semigroup structure of the alphabet
\[
\mathfrak{M}=\left\{[ \begin{array}{c}
s \\
r 
\end{array} ] \ ;\  (s, r) \in  \mathbb{Z}\times \mathbb{R}^{+*} \right\rbrace.
\]
They propose then a directional regularization of MZV defined on words
\[ Z([ \begin{array}{c}
s_1 \\
r_1 
\end{array} ]\dots [ \begin{array}{c}
s_k \\
r_k 
\end{array} ] ;\varepsilon) =\sum_{n_1>\dots>n_k>0}\frac{e^{n_1 r_1 \varepsilon}\dots e^{n_k r_k \varepsilon}}{n_1^{s_1}\dots n_k^{s_k}}
\]
that defines a character on $\mathcal{H}_{\mathfrak{M}}$ with values in an algebra of Laurent series. The formula they give for the Birkhoff decomposition (Theorem 3.8) coincide equation (\ref{eq:wordbirk}). 
\end{example}
\begin{example}[Rooted ladders]

As a toy model for applications in physics \cite[section 4.2]{EFP1} considers a character on the polynomial commutative Hopf algebra $\mathcal{H}^{\text{lad}}$ of ladder trees. If the ladder tree with $n$ nodes is $t_n$, then
\[
\Delta(t_n)=t_n \otimes 1+\sum_{k=1}^{n-1} t_k \otimes t_{n-k} +1 \otimes  t_n.
\]
It is a matter of fact to identify the coalgebra structure of $\mathcal{H}^{\text{lad}}$ with the tensor deconcatenation coalgebra $T(\{x\})$ over an alphabet with one letter, where $t_n$ corresponds to the word $\underbrace{x \dots x}_n$. Formula (\ref{eq:wordbirk}) can be applied to the character mapping the tree $t_n$ to an $n$-fold Chen's iterated integral defined recursively by 
\[
\psi(p ; \varepsilon, \mu)(t_n)= \mu^{\varepsilon}\int_p^{\infty} \psi(x ; \varepsilon, \mu)(t_{n-1})\frac{dx}{x^{1+\varepsilon}}=\frac{e^{-n\varepsilon \log(p/\mu)}}{n!\varepsilon^n}=f_n(\varepsilon)
\]
with values in the Laurent series in $\varepsilon$.
We get for the couterterms:
\begin{equation}
\psi_-(p ; \varepsilon, \mu)(t_n))=\sum_{n_1+\dots+n_t=n \atop t\geq 1 \ ,\ n_i >0}(-1)^{t} (-1)^{t} p_- ( p_- ( \ldots ( p_- ( f_{n_1}(\varepsilon)) f_{n_2}(\varepsilon) ) \ldots ) f_{n_t}(\varepsilon))
\end{equation}
\end{example}
\begin{example}[Differential equations]

When dealing with differential equations and associated diffeomorphisms (flow, conjugacy map), characters on shuffle Hopf algebras appear almost naturally. For instance, such characters correspond to:
\begin{itemize}
\item the coefficients of word series in \cite{MSS},
\item "symmetral moulds" in mould calculus (see \cite{FM,snag})
\item or Chen's iterated integrals (see for instance \cite{kre, manchon}).
\end{itemize}

Let us just give the example of a simple differential equation related to mould calculus (see \cite{men06}). Let $b ( x, y )=\sum_{n\geq 0}x^n b_n(y) \in y^2 \mathbbm{C}[ [ x, y ] ]$ and $d \in \mathbbm{N}$. If one looks for a formal identity tangent diffeomorphism $\varphi ( x, y )$ in $y$, with
coefficients in $\mathbbm{C}[ [ x ] ]$ such that, if $y$ is a
solution of
\begin{equation*}
  ( E_{b, d} ) \hspace{2em} x^{1 - d} \partial_x y = b ( x, y )
\end{equation*}
then $z = \varphi ( x, y )$ is a solution of
\begin{equation*}
  ( E_{0, d} ) \hspace{2em} x^{1 - d} \partial_x z = 0.
\end{equation*}
One can try to compute this diffeomorphism as a "mould series":
\begin{equation}\label{eq:moulddiff}
\varphi_d ( x, y )=y + \sum_{s \geq 1} \sum_{n_1, \ldots n_s \in
   \mathbbm{N}} V_d(n_1, \ldots, n_s) \mathbbm{B}_{n_s} \ldots
   \mathbbm{B}_{n_1}.y \quad (\mathbbm{B}_n = b_n ( y ) \partial_y )
\end{equation}
where $V_d$ is a character on the shuffle algebra $T(\mathbbm{N})$, with values in $\mathbbm{C}[ [ x ] ]$. Whenever $d$ is a positive integer, this character can be computed and for any word $( n_1, \ldots, n_s )$

\begin{equation}
V_d(n_1, \ldots, n_s) =  \frac{( - 1 )^s x^{n_1 + \ldots + n_s + s
       d}}{( \check{n}_1 + d ) ( \check{n}_2 + 2 d ) \ldots ( \check{n}_s + s
       d )} \quad ( \check{n}_i = n_1 + \ldots + n_i ).
\end{equation}  

The map $\varphi_d ( x, y )\in \mathbb{C}[[x,y]]$ is then well defined and conjugates $( E_{b, d} )$ to $(E_{0, d} )$. For $d=0$, there may be divisions by $0$ and, in this case, one can consider $d=\varepsilon$ as a real parameter and use the expansion $x^{\varepsilon}=\sum \frac{(\varepsilon \log x))^n}{n!}$ so that the character $V_{\varepsilon}$ has its values in $\mathfrak{B}[ [
\varepsilon ] ] [ \varepsilon^{- 1} ]$ where $\mathfrak{B}=\mathbbm{C}[[\log x,
x]$. If one uses the same formula (\ref{eq:wordbirk}) to perform the Birkhoff decomposition, the regular character $V_{\varepsilon,+}$, evaluated at $\varepsilon=0$ allows to find a diffeomorphism (as in equation (\ref{eq:moulddiff})) that conjugates $x \partial_x
  y = b ( x, y )$ to $x \partial_x z = 0$ with a price to
pay : it contains monomials in $x$ and $\log x$. See \cite{men06} for details.
\end{example}

%\section{Conclusion}
%
%\textit{The same could also be done with characters on Connes-Kreimer Hopf algebras, that appear as well in pQFT, see \cite{ck1}, as in dynamics, see \cite{FM} or in numerical analysis, in relation with B-series.}

Not also that the same ideas can be used for the the even-odd factorization of characters
in combinatorial Hopf algebras (see {\cite{ag1}}, {\cite{ag2}} and
{\cite{guo}}).

%
%\begin{acknowledgement}
%If you want to include acknowledgments of assistance and the like at the end of an individual chapter please use the \verb|acknowledgement| environment -- it will automatically render Springer's preferred layout.
%\end{acknowledgement}
%%
%
%\begin{equation}
%a \times b = c
%\end{equation}


\begin{thebibliography}{99.}

  \bibitem{ag1} M. Aguiar, N. Bergeron, and F. Sottile. Combinatorial Hopf algebras and generalized Dehn-Sommerville relations. {Compos. Math.} 
  142, 1 (2006), 1--30.
  
\bibitem{ag2} M. Aguiar and S.~K. Hsiao. Canonical
  characters on quasi-symmetric functions and bivariate Catalan numbers.
 {Electron. J. Combin.} 11, 2 (2004/06). Research Paper 15, 34 pp.
  (electronic).

\bibitem{BF} C. Brouder, A. Frabetti and C. Krattenthaler. Non-Commutative Hopf Algebra of Formal Diffeomorphisms. Advances in Mathematics 200, 2 (2006), 479--524.
 
   \bibitem{hz} Y Bruned, M Hairer, L Zambotti, Algebraic renormalisation of regularity structures, arXiv preprint arXiv:1610.08468, 2016.
   
 \bibitem{Cartier1:2000}
P. Cartier, Vinberg algebras, Lie groups and combinatorics. Clay Mathematics Proceedings. Quanta of Maths, 11 (2010), 107-126.
 
\bibitem{Cartier2} P. Cartier,
  A primer of Hopf algebras, Frontiers in number theory, physics, and geometry II, Springer Berlin Heidelberg, 537-615 (2017).
 

  \bibitem{ck1} A. Connes and D. Kreimer. Renormalization in
  quantum field theory and the Riemann-Hilbert problem. I: The Hopf algebra
  structure of graphs and the main theorem. Commun. Math.
  Phys. 210, 1 (2000), 249--273.
  
  \bibitem{ck2} A. Connes and D. Kreimer. Renormalization in
  quantum field theory and the Riemann-Hilbert problem. II: The
  $\beta$-function, diffeomorphisms and the renormalization group. 
  Commun. Math. Phys. 216, 1 (2001), 215--241.
  
  \bibitem{CM} A. Connes and M. Marcolli. From physics to number theory via noncommutative geometry. Frontiers in number theory, physics, and geometry. I, Springer, Berlin, 2006, 269--347. 

\bibitem{kur}K. Ebrahimi-Fard, L.~Guo and D. Kreimer.
 Integrable renormalization. I: The ladder case. {J. Math. Phys.} 45, 10 (2004), 3758--3769.
 
\bibitem{guo}K. Ebrahimi-Fard, L.~Guo, and D. Manchon. Birkhoff type decompositions and the Baker-Campbell-Hausdorff recursion. {Comm. Math. Phys.} 267, 3 (2006), 821--845.

  \bibitem{KGP} K. Ebrahimi-Fard, J. Gracia-Bondia and F. Patras. A Lie theoretic approach to renormalization. Comm. Math. Phys. 276 (2007), 519--549.

%  \bibitem{mat} K. Ebrahimi-Fard, J.~M. Gracia-Bondia, L.~Guo, and J. C. Várilly. Combinatorics of renormalization as matrix calculus. {Phys. Lett. B} 632, 4 (2006), 552--558.
  
  \bibitem{emp}K. Ebrahimi-Fard, D. Manchon and F. Patras. A noncommutative Bohnenblust-Spitzer identity for Rota-Baxter algebras solves Bogolioubov's recursion . Journal of Noncommutative Geometry , Vol. 3, Issue 2 (2009), 181-222. 
  
\bibitem{EFP1} K. Ebrahimi-Fard and F. Patras. Exponential Renormalization Annales Henri Poincaré 11, 5 (2010), 943--971.
  
  
  \bibitem{EFP2} K. Ebrahimi-Fard and F. Patras, Exponential Renormalization II: Bogoliubov's R-operation and momentum subtraction schemes J. Math. Phys. 53, 8 (2012), 15pp. 
  
  
  
 \bibitem{snag} J. Ecalle, Singularit\'es non abordables par la g\'eom\'etrie. (French) [Singularities
that are inaccessible by geometry] Ann. Inst. Fourier 42, 1--2 (1992), 73--164.
 
 \bibitem{FM}
F.~Fauvet, F.~Menous, Ecalle's arborification-coarborification transforms and
Connes-Kreimer Hopf algebra. To appear in {Ann. Sci. \'Ec. Norm. Sup\'er.} (2017) 
51~pp. 
  \bibitem{fig} H. Figueroa and J.~M. Gracia-Bondia. Combinatorial Hopf algebras in quantum field theory. I.{Rev. Math. Phys.} 17, 8 (2005), 881--976.


\bibitem{fpQSh}L. Foissy and F. Patras, Lie theory for quasi-shuffle algebras. ArXiV, 2016.

\bibitem{MF} A. Frabetti and D. Manchon. Five Interpretations of Faà Di Bruno's Formula. In Dyson-Schwinger Equations and Faà Di Bruno Hopf Algebras in Physics and Combinatorics, edited by European Mathematical Society, 5-65. Strasbourg, France, 2011.



\bibitem{GZ} L. Guo and B. Zhang. Renormalization of Multiple Zeta Values. Journal of Algebra 319, 9 (2008), 3770--3809.

  \bibitem{hai}M. Hairer, A theory of regularity structures, Inventiones mathematicae, 198 (2), (2014), 269--504.
  

 
 \bibitem{hof} M. E. Hoffman. Quasi-shuffle products. {J. Algebr. Comb.} 11, 1 (2000), 49--68.
  \bibitem{hi} M. E. Hoffman and K. Ihara. Quasi-shuffle products revisited. J. Algebra, 481, (2017), 293--326.
  
 \bibitem{Kandihar}R.L. Karandikar,  Multiplicative decomposition of non-singular matrix valued continuous semimartingales. The Annals of Probability, 10(4), (1982), 1088--1091.
 
 %\bibitem{Kandihar2}R.L. Karandikar, Multiplicative decomposition of nonsingular matrix valued semimartingales. In Séminaire de Probabilit\'es XXV (pp. 262-269). Springer, Berlin, Heidelberg (1991).
  
  
  \bibitem{kre}D. Kreimer. Chen's iterated integral
  represents the operator product expansion. {Adv. Theor.
  Math. Phys.} 3, 3 (1999), 627--670.
  
  \bibitem{maj}S. Majid. {Foundations of quantum
  group theory.} Cambridge Univ. Press., 1995.
  
  \bibitem{manchon}D. Manchon and S. Paycha. Shuffle
  relations for regularised integrals of symbols,  Comm. Math. Phys. 270,  (2007), 13--51.


  \bibitem{men}F. Menous. On the stability of some groups of formal diffeomorphisms by the Birkhoff decomposition. Adv. Math. 216, 1 (2007), 1--28.
  
 \bibitem{men06} F. Menous. Formal differential equations and renormalization. Connes, Alain (ed.) et al., Renormalization and Galois theories. European Mathematical Society, IRMA Lectures in Mathematics and Theoretical Physics 15 (2009) 229--246.
  
  \bibitem{men11} F. Menous. Formulas for the Connes-Moscovici Hopf Algebra. In Combinatorics and Physics. 539, 269--28585. Contemporary Mathematics. Ebrahimi-Fard, Kurusch (ed.) et al., 2011.

\bibitem{men13} F. Menous. From Dynamical Systems to Renormalization. Journal of Mathematical Physics 54, 9 (2013), 24pp.

\bibitem{MP} F. Menous and F. Patras, Logarithmic Derivatives and Generalized Dynkin Operators. Journal of Algebraic Combinatorics: Volume 38, Issue 4 (2013), Page 901--913 
  
  \bibitem{MSS} A. Murua and J.M. Sanz-Serna. Computing Normal Forms and Formal Invariants of Dynamical Systems by Means of Word Series. Nonlinear Analysis: Theory, Methods and Applications 138 (2016), 326--345.
  
  \bibitem{Pat} F. Patras,
L'alg\`ebre des descentes d'une big\`ebre gradu\'ee. J. Algebra 170, 2 (1994), 547-566.


\bibitem{patrasC} F. Patras, Dynkin operators and renormalization group actions in pQFT. in: Vertex Operator Algebras and Related Areas, Eds M. Bergvelt, G. Yamskulna, W. Zhao, Contemp. Math. vol. 497 (2009), 169--184.


\bibitem{schutz} M.-P. Sch\"utzenberger, Sur une propri\'et\'e combinatoire des alg\`ebres de Lie libres pouvant \^etre utilis\'ee dans un probl\`eme de math\'ematiques appliqu\'ees, S\'eminaire Dubreil--Jacotin Pisot (Alg\`ebre et th\'eorie des nombres), 1958/59.

\bibitem{sweed}M.E. Sweedler. {Hopf algebras}. W.A. Benjamin, Inc., 1969.


\end{thebibliography}
\end{document}